\documentclass[reqno,11pt]{amsart}

\usepackage{a4wide}
\usepackage[T1]{fontenc}
\usepackage{ae,aecompl}
\usepackage[english]{babel}
\usepackage[latin1]{inputenc}
\usepackage{enumerate}
\usepackage{latexsym}
\usepackage{amsmath,amsthm,amsfonts}
\usepackage{xr}
\usepackage{graphicx}
\usepackage{stmaryrd}
\usepackage{mathtools}
\usepackage{braket}
\usepackage{pgfplots}
\usepackage{textcomp}
\usepackage{bbm}

\newtheorem{satz}{Theorem}[section]
\newtheorem{cor}[satz]{Corollary}
\newtheorem{lem}[satz]{Lemma}
\newtheorem{prop}[satz]{Proposition}
\newtheorem{defi}[satz]{Definition}
\newtheorem{bem}[satz]{Remark}
\newtheorem{bsp}[satz]{Example}

\numberwithin{equation}{section}

\newcommand{\sgn}{\operatorname{sgn}}
\newcommand{\erf}{\operatorname{erf}}
\renewcommand{\Re}{\operatorname{Re}}
\renewcommand{\Im}{\operatorname{Im}}

\newcommand{\Arg}{\operatorname{Arg}}
\newcommand{\rank}{\operatorname{rank}}

\newcommand{\vvect}[2]{\left(\begin{array}{c} #1 \\ #2 \end{array}\right)}
\newcommand{\mmatrix}[4]{\left(\begin{array}{cc} #1 & #2 \\ #3 & #4 \end{array}\right)}

\title[Green's Function for the Schr\"odinger Equation]
{\bf Green's Function for the Schr\"odinger Equation with a Generalized Point Interaction
and Stability of Superoscillations}

\author[Yakir Aharonov]{Yakir Aharonov}
\address{(YA) Schmid College of Science and Technology, Chapman University, Orange 92866, CA, USA}
\email{aharonov@chapman.edu}

\author[Jussi Behrndt]{Jussi Behrndt}
\address{(JB) Institut f\"ur Angewandte Mathematik, Technische Universit\"at Graz, Steyrergasse 30, 8010 Graz, Austria}
\email{behrndt@tugraz.at}
\address{(JB) Department of Mathematics, Stanford University, 450 Jane Stanford Way, Stanford CA 94305-2125, US}
\email{jbehrndt@stanford.edu}

\author[Fabrizio Colombo]{Fabrizio Colombo}
\address{(FC) Politecnico di Milano, Dipartimento di Matematica, Via E. Bonardi, 9, 20133 Milano, Italy}
\email{fabrizio.colombo@polimi.it}

\author[Peter Schlosser]{Peter Schlosser}
\address{(PS) Institut f\"ur Angewandte Mathematik, Technische Universit\"at Graz, Steyrergasse 30, 8010 Graz, Austria}
\email{schlosser@tugraz.at}

\begin{document}

\begin{abstract}\vspace{1cm}
In this paper we study the time dependent Schr\"odinger equation with all possible self-adjoint singular interactions located at the origin,
which include the $\delta$
and $\delta'$-potentials as well as boundary conditions of Dirichlet, Neumann, and  Robin type as particular cases.
 We derive an explicit representation of the time dependent
Green's function and give a mathematical rigorous meaning to the corresponding integral for holomorphic initial conditions, using Fresnel integrals.
Superoscillatory functions  appear in the context of weak measurements in quantum mechanics
and are naturally treated as holomorphic entire functions.
As an application of the Green's function we study the stability and oscillatory
properties of the solution of the Schr\"odinger equation subject to a generalized point interaction
when the initial datum is a superoscillatory function.
\end{abstract}

\maketitle


%\vskip 1cm

\tableofcontents

\section{Introduction}

\medskip

The main purpose of this paper is to study the Green's function of the time dependent Schr\"{o}dinger equation subject to general self-adjoint
point interactions located at the origin, and to prove stability results for the solutions corresponding to
superoscillating initial data.
As a consequence of our detailed analysis we also obtain an explicit expression and
asymptotic expansion of the time dependent plane wave solution,
which allows to discuss the oscillatory properties of the time evolution of superoscillations under generalized point interactions.

Strongly localized potentials, also called pseudo-potentials or nowadays better known as $\delta$-potentials,
were already considered by Kronig and Penney in \cite{KrPe31} and Fermi in \cite{Fe36}. Heuristically speaking,
these $\delta$-potentials are represented by the Hamiltonian
\begin{equation}\label{Eq_Hamiltonian_delta}
H=-\Delta+\sum_{y\in Y}c_y\,\delta(x-y),
\end{equation}
where $c_y\,\delta(x-y)$ is a point source of strength $c_y$ located at the point $y\in\mathbb R^d$, $d\geq 1$.
The $\delta$-potentials may form a discrete set, e.g., a periodic lattice
$Y=\mathbb{Z}^d$, or a single point $Y=\{0\}$. The rigorous mathematical
meaning of the Hamiltonian \eqref{Eq_Hamiltonian_delta} was given only much later by Berezin and Faddeev in \cite{BeFa61}.

In this paper we will restrict ourselves to a single point interaction in $\mathbb R$
and hence assume  $Y=\{0\}$ and $d=1$ from now on.
In this context $H$ in \eqref{Eq_Hamiltonian_delta} is defined as a proper self-adjoint extension of the symmetric
operator $-\Delta$ on $C^\infty_0(\mathbb{R}\setminus\{0\})$ which corresponds to
interface (or jump) conditions at the origin of the form
\begin{equation}\label{del}
 \begin{split}
 u(0^+)&=u(0^-), \\
u'(0^+)-u'(0^-)&=c_0\,u(0);
 \end{split}
\end{equation}
a detailed discussion
can be found in the standard monograph \cite{ALB}.
Besides the $\delta$-potential also other types of self-adjoint interface conditions can be treated (see \cite{BrJe01,Ca93,ChHu93,ExGr99,KK03,RoTa96,Se86}
and \cite{BLL13,ER16,MPS16} for
interactions on hypersurfaces),
among them so-called $\delta'$-potentials and further generalizations, as well as
decoupled systems with Dirichlet, Neumann, or Robin conditions. There are various ways to describe the complete family of self-adjoint interface conditions
at the origin and for our purposes it is convenient to use the parametrization
\begin{equation}\label{Eq_General_jump_condition}
(I-J)\vvect{u(0^+)}{u(0^-)}=i(I+J)\vvect{u'(0^+)}{-u'(0^-)},
\end{equation}
with unitary $2\times 2$-matrices $J$ (see Example~\ref{bsp_Delta} for identifying \eqref{del} as a special case of \eqref{Eq_General_jump_condition}).  To be more precise:
The class of jump conditions \eqref{Eq_General_jump_condition} coincides with the class of self-adjoint interface conditions
at the point $0$. In other words, each unitary matrix $J\in\mathbb{C}^{2\times 2}$  leads to a
self-adjoint realization of the Laplacian in $L^2(\mathbb{R})$ with a generalized point interaction supported at the point $0$, and conversely,
for each self-adjoint Laplacian in $L^2(\mathbb{R})$ with a generalized point interaction supported at the point $0$ there exists a
unitary matrix $J\in\mathbb{C}^{2\times 2}$ such that the interface condition has the form \eqref{Eq_General_jump_condition}; cf. \cite[Chapter 2.2]{BHS20}.

An important problem we study in this paper is the time dependent Schr\"odinger equation
 with holomorphic initial datum $F$ subject to
a general self-adjoint singular interaction supported at the origin, that is, we consider
\begin{subequations}\label{Eq_Schroedinger}
\begin{align}
i\frac{\partial}{\partial t}\Psi(t,x)&=-\frac{\partial^2}{\partial x^2}\Psi(t,x),\hspace{3.1cm}t>0,\,x\in\mathbb{R}\setminus\{0\}, \label{Eq_Schroedinger_1} \\
(I-J)\vvect{\Psi(t,0^+)}{\Psi(t,0^-)}&=i(I+J)\vvect{\frac{\partial}{\partial x}\Psi(t,0^+)}{-\frac{\partial}{\partial x}\Psi(t,0^-)},\qquad t>0, \label{Eq_Schroedinger_2} \\
\Psi(0^+,x)&=F(x),\hspace{4.4cm} x\in\mathbb{R}\setminus\{0\}. \label{Eq_Schroedinger_3}
\end{align}
\end{subequations}
It will be shown in Section~\ref{sec_Greens_function_of_singular_interactions} that the corresponding Green's function is given by
\begin{equation}\label{Eq_Greenfunction}
\begin{split}
G(t,x,y)=&\left(\mu_+^{(x,y)}\Lambda\left(\frac{|x|+|y|}{2\sqrt{it}}+\omega_+\sqrt{it}\right)+\mu_-^{(x,y)}\Lambda\left(\frac{|x|+|y|}{2\sqrt{it}}+\omega_-\sqrt{it}\right)\right)e^{-\frac{(|x|+|y|)^2}{4it}} \\
&+\frac{1}{2\sqrt{i\pi t}}\left(\mu_0^{(x,y)}e^{-\frac{(|x|+|y|)^2}{4it}}+e^{-\frac{(x-y)^2}{4it}}\right),\qquad t>0,\,x,y\in\mathbb{R}\setminus\{0\},
\end{split}
\end{equation}
where the entire function $\Lambda$ is defined in \eqref{Eq_Lambda} and the coefficients $\mu_\pm$, $\mu_0$, and $\omega_\pm$ are explicitly
determined in terms of the entries of the unitary matrix $J$ in the jump condition \eqref{Eq_Schroedinger_2}; cf. Theorem~\ref{satz_Green_function}, the examples in Section~\ref{Special cases }, and
\cite{AlBrDa94,AlBrDa95,GaSc86,Ma89,Se71} for related results.
Using the Green's function \eqref{Eq_Greenfunction} the solution
$\Psi$ of \eqref{Eq_Schroedinger} can be written as the integral
\begin{equation}\label{Eq_WavefunctionHH}
\Psi(t,x)=\int_\mathbb{R}G(t,x,y)F(y)dy,\qquad t>0,\,x\in\mathbb{R}\setminus\{0\}.
\end{equation}
While this integral is well defined for compactly supported continuous functions $F$, one has difficulties in making sense of \eqref{Eq_WavefunctionHH}
already for plane waves $F(x)=e^{ikx}$. A mathematical rigorous analysis of this issue for a certain class of holomorphic functions with
growth condition is provided in Section \ref{sec_Solution_of_the_Schroedinger_equation}, where the main tool is the Fresnel integral approach.

\medskip
The general results in Section~\ref{sec_Greens_function_of_singular_interactions} and Section~\ref{sec_Solution_of_the_Schroedinger_equation}
are applied to superoscillations in Section~\ref{sec_Time_evolution_of_superoscillations}.
Superoscillating functions are band-limited functions that can oscillate faster than their fastest Fourier component.
They appear in quantum mechanics as results of weak measurements and, in particular, their time evolution under the Schr\"odinger equation is of crucial importance,
see \cite{aav,abook,av,b5}.
For a rigorous treatment of this subject we refer to \cite{ABCS19,acsst4,acsst1,acsst3,acsst6,JFAA,AOKI,Jussi,harmonic} and \cite{acsst5}.
These kind of functions (or sequences) also appear in antenna theory \cite{TDFG} (see also \cite{MBANTENNA}) and
various applications in optics were studied by M.V. Berry and many others, see, e.g.,
\cite{berry2,berry,berry-noise-2013,BerryMILAN,b1,b4,kempf1,kempf2,leeferreira2,lindberg}.
More information can also be found in the introductory papers \cite{QS1,QS3,QS2,kempfQS} and in the \textit{Roadmap on superoscillations} \cite{Be19}.

A weak measurement of a quantum observable represented by the self-adjoint operator $A$,
involving a pre-selected state $\psi_0$ and a post-selected state $\psi_1$, leads to the weak value
$$
A_{weak}:=\frac{(\psi_1,A\psi_0)}{(\psi_1,\psi_0)}\in\mathbb C,
$$
where the real part of $A_{weak}$ can be interpreted as the shift and the imaginary part as the
momentum of the pointer recording the measurement.
An important feature of the weak measurement is that, in contrast with strong measurements $A_{strong}:=(\psi, A\psi)$,
the real part of $A_{weak}$ may become very large when the states $\psi_0$ and $\psi_1$
are almost orthogonal; this leads to superoscillations.
A typical superoscillatory function is
\begin{equation}\label{Eq_Fn}
F_n(x,k)=\sum_{l=0}^nC_l(n,k)e^{i(1-\frac{2l}{n})x},\qquad x\in\mathbb R,
\end{equation}
where $|k|>1$ and
\begin{equation*}
C_l(n,k)=\vvect{n}{l}\left(\frac{1+k}{2}\right)^{n-l}\left(\frac{1-k}{2}\right)^l.
\end{equation*}
If we fix $x\in\mathbb{R}$ and let $n$ tend to infinity, we obtain $\lim_{n\to\infty}F_n(x,k)=e^{ikx}$
%\begin{equation}\label{Eq_Fn_convergence}
%\lim_{n\to\infty}F_n(x,k)=e^{ikx},
%\end{equation}
uniformly for $x$ in compact subsets of $\mathbb{R}$. The notion superoscillations comes
from the fact that the frequencies $(1-\frac{2l}{n})$ in (\ref{Eq_Fn}) are in modulus bounded by 1,
but the frequency $k$ of the limit function can be arbitrarily large.

\medskip
As a consequence of the representation \eqref{Eq_WavefunctionHH} of the solution of
the Schr\"odinger equation subject to
a general self-adjoint singular interaction
we ask: When does a convergent sequence of initial conditions
\begin{equation}\label{aq}
\lim\limits_{n\to\infty}F_n=F
\end{equation}
also lead to a convergent sequence of solutions
\begin{equation}\label{ay}
\lim\limits_{n\rightarrow\infty}\Psi(t,x;F_n)=\Psi(t,x;F),
\end{equation}
and which type of convergence should be considered in \eqref{aq} and \eqref{ay}$\,$? Our abstract result Theorem \ref{satz_Psi_convergence}, which is also the bridge
to investigate the time evolution of superoscillations in Section \ref{sec_Time_evolution_of_superoscillations}, shows that \eqref{ay} holds
uniformly on compact subsets of $(0,\infty)\times\mathbb{R}$,
whenever the sequence $(F_n)_n$ satisfy some exponential boundedness conditions and  the convergence in \eqref{aq} is such that
\[
\lim\limits_{n\rightarrow\infty}\sup\limits_{z\in S_\alpha\cup(-S_\alpha)}\big|F_n(z)-F(z)\big|e^{-C|z|}=0
\]
for some $C\geq 0$ and certain sectors $S_\alpha$ and $-S_\alpha$ in the complex plane; cf. Section~\ref{sec_Solution_of_the_Schroedinger_equation} for more details.
These abstract assumptions are in accordance with the convergence properties of (holomorphic extensions of) superoscillating functions
in spaces of entire functions with exponential growth that
have been clarified just in the recent years, see \cite{CSSYgenfun}.
The case of superoscillatory initial data is then discussed in Corollary~\ref{cor_Stability_of_superoscillations} and the explicit form,
oscillatory behaviour, and long time asymptotics
of the corresponding limit in \eqref{ay} are provided in Proposition~\ref{prop_Plane_wave_solution} and Theorem~\ref{satz_Asymptotics}.

\medskip\noindent
{\bf Acknowledgement.}
J.\,B. gratefully acknowledges support for the Distinguished Visiting Austrian Chair at Stanford University by the Europe Center and the
Freeman Spogli Institute for International Studies.

\section{Green's function for the Schr\"odinger equation with a generalized point interaction}\label{sec_Greens_function_of_singular_interactions}

In this section we derive the Green's function of the time dependent Schr\"odinger equation (\ref{Eq_Schroedinger}) with a generalized singular interaction located at the origin.
That is, we construct a function $G$ which depends on the matrix $J$, such that the solution $\Psi$ of \eqref{Eq_Schroedinger} can be written in the form
\begin{equation}\label{Eq_Green_function_integral_sec2}
\Psi(t,x)=\int_\mathbb{R}G(t,x,y)F(y)dy,\qquad t>0,\,x\in\mathbb{R}\setminus\{0\}.
\end{equation}
In Section \ref{sec_Solution_of_the_Schroedinger_equation} we shall clarify for which initial conditions $F$ and in which sense this integral is understood.
Here, we only want to derive the explicit form and some properties of the Green's function $G$ itself.

\medskip

We start by defining the entire function
\begin{equation}\label{Eq_Lambda}
\Lambda(z):= e^{z^2}(1-\erf(z)),\qquad z\in\mathbb{C},
\end{equation}
where $\erf(z)=\frac{2}{\sqrt{\pi}}\int_0^ze^{-\xi^2}d\xi$ is the well known error function. Some important properties of
this function are collected in the following lemma; cf. \cite[Lemma 3.1]{Jussi}.

\begin{lem}\label{lem_Lambda_Properties}
The function $\Lambda$ in \eqref{Eq_Lambda} has the following properties:

\begin{enumerate}
\item[{\rm (i)}] The function $\Lambda$ satisfies the differential equation
\begin{equation}\label{Eq_Lambda_Derivative}
\frac{d}{dz}\Lambda(z)=2z\Lambda(z)-\frac{2}{\sqrt{\pi}},\qquad z\in\mathbb{C}.
\end{equation}

\item[{\rm (ii)}] The value of the function $\Lambda$ at $-z$ is given by
\begin{equation}\label{Eq_Lambda_Negative}
\Lambda(-z)=2e^{z^2}-\Lambda(z),\qquad z\in\mathbb{C}.
\end{equation}

\item[{\rm (iii)}] The absolute value of $\Lambda(z)$ can be estimated by
\begin{equation}\label{Eq_Lambda_Estimate}
|\Lambda(z)|\leq\Lambda(\Re(z)),\qquad z\in\mathbb{C}.
\end{equation}

\item[{\rm (iv)}] The function $\Lambda$ is monotonically decreasing on $\mathbb{R}$ and asymptotically on
$\mathbb{C}$ one has
\begin{equation}\label{Eq_Lambda_Asymptotic}
\Lambda(z)=\left\{\begin{array}{ll} \mathcal{O}\big(\frac{1}{|z|}\big), & \text{if }\Re(z)\geq 0, \\ 2e^{z^2}+\mathcal{O}\big(\frac{1}{|z|}\big), & \text{if }\Re(z)\leq 0, \end{array}\right.\quad\text{as }|z|\rightarrow\infty.
\end{equation}

\item[{\rm (v)}] For all $a>0$ and $b,c\in\mathbb{C}$ one has the integral identities
\begin{subequations}
\begin{align}
\int_0^\infty e^{-ax^2-bx}dx&=\frac{\sqrt{\pi}}{2\sqrt{a}}\Lambda\Big(\frac{b}{2\sqrt{a}}\Big), \label{Eq_Lambda_Integral} \\
\int_0^\infty e^{-ax^2-bx}\Lambda\big(\sqrt{a}\,x+c\big)dx&=-\frac{1}{2\sqrt{a}}\left\{\begin{array}{ll} \frac{\Lambda(c)-\Lambda\big(\frac{b}{2\sqrt{a}}\big)}{c-\frac{b}{2\sqrt{a}}}, & \text{if }c\neq\frac{b}{2\sqrt{a}}, \\ \Lambda'(c), & \text{if }c=\frac{b}{2\sqrt{a}}. \end{array}\right. \label{Eq_Lambda_Integral_2}
\end{align}
\end{subequations}
\end{enumerate}
\end{lem}

\begin{proof}
(i) and (ii) are contained in \cite[Lemma 3.1]{Jussi}.
\medskip

(iii) Using $\int_0^\infty e^{-\xi^2}d\xi=\frac{\sqrt{\pi}}{2}$ in the definition \eqref{Eq_Lambda} gives
\begin{equation}\label{Eq_Lambda_Property_1}
\Lambda(z)=\frac{2}{\sqrt{\pi}}e^{z^2}\left(\int_0^\infty e^{-\xi^2}d\xi-\int_0^ze^{-\xi^2}d\xi\right).
\end{equation}
Now we use that the complex integral over the entire function $e^{-\xi^2}$ is path independent and that
$\lim_{x\to\infty}\int_x^{x+z}e^{-\xi^2}d\xi=0$. Hence, the two integrals on the right-hand side of \eqref{Eq_Lambda_Property_1}
can be replaced by a path integral from $z$ to $\infty$, parallel to the real axis. This gives
\begin{equation}\label{Eq_Lambda_integral_representation}
\Lambda(z)=\frac{2}{\sqrt{\pi}}e^{z^2}\int_0^\infty e^{-(z+s)^2}ds=\frac{2}{\sqrt{\pi}}\int_0^\infty e^{-s^2-2zs}ds.
\end{equation}
This representation can now be used to estimate the absolute value
\begin{equation*}
|\Lambda(z)|\leq\frac{2}{\sqrt{\pi}}\int_0^\infty e^{-s^2-2\Re(z)s}ds=\Lambda(\Re(z)).
\end{equation*}

(iv) The monotonicity is a direct consequence of the representation \eqref{Eq_Lambda_integral_representation}
and the asymptotics were shown in \cite[Lemma 3.1]{Jussi}.

\medskip

(v) Substituting $t=\frac{s}{\sqrt{a}}$ in the integral \eqref{Eq_Lambda_integral_representation} gives
\begin{equation*}
\Lambda(z)=\frac{2\sqrt{a}}{\sqrt{\pi}}\int_0^\infty e^{-at^2-2z\sqrt{a}\,t}dt,
\end{equation*}
which is exactly \eqref{Eq_Lambda_Integral} if we evaluate at $z=\frac{b}{2\sqrt{a}}$. In order to check \eqref{Eq_Lambda_Integral_2} we first use \eqref{Eq_Lambda_Derivative} to obtain the primitive
\begin{equation*}
e^{-ax^2-bx}\Lambda\big(\sqrt{a}\,x+c\big)=\frac{1}{2\sqrt{a}}\frac{d}{dx}e^{-ax^2-bx}\frac{\Lambda\big(\sqrt{a}\,x+c\big)-\Lambda\Big(\sqrt{a}\,x+\frac{b}{2\sqrt{a}}\Big)}{c-\frac{b}{2\sqrt{a}}}.
\end{equation*}
The assertion on the integral in \eqref{Eq_Lambda_Integral_2} now simply follows by evaluating at $x=0$ and $x\to\infty$; observe that
by \eqref{Eq_Lambda_Asymptotic} the limit $x\to\infty$ vanishes. Similarly, also in the case $b=2\sqrt{a}\,c$ we get the primitive
\begin{equation*}
e^{-ax^2-2\sqrt{a}\,cx}\Lambda\big(\sqrt{a}\,x+c\big)=\frac{1}{2\sqrt{a}}\frac{d}{dx}e^{-ax^2-bx}\Lambda'\big(\sqrt{a}\,x+c\big)
\end{equation*}
and we also get the second case of the integral \eqref{Eq_Lambda_Integral_2} by evaluating the primitive at $x=0$ and $x\to\infty$.
\end{proof}

Using \eqref{Eq_Lambda} we now define for every $t>0$, $x\in\mathbb{R}\setminus\{0\}$, $z\in\mathbb{C}$, and $\omega\in\mathbb{R}$, the functions
\begin{subequations}\label{Eq_Part_Greenfunction}
\begin{align}
G_0(t,x,z)&:=\frac{1}{2\sqrt{i\pi t}}e^{-\frac{(|x|+z)^2}{4it}}, \label{Eq_Part_Greenfunction0} \\
G_1(t,x,z)&:=\Lambda\left(\frac{|x|+z}{2\sqrt{it}}+\omega\sqrt{it}\right)e^{-\frac{(|x|+z)^2}{4it}}, \label{Eq_Part_Greenfunction1} \\
G_\text{free}(t,x,z)&:=\frac{1}{2\sqrt{i\pi t}}e^{-\frac{(x-z)^2}{4it}}, \label{Eq_Part_Greenfunctionfree}
\end{align}
\end{subequations}
which will appear as components of the Green's function \eqref{Eq_Greenfunction} later on.
In the following preparatory lemma we check that each of these components is a solution of the free Schr\"odinger equation on $\mathbb{R}\setminus\{0\}$.

\begin{lem}\label{lem_Part_Greenfunction_DGL}
For every $t>0$, $x\in\mathbb{R}\setminus\{0\}$, $z\in\mathbb{C}$, the functions in \eqref{Eq_Part_Greenfunction} satisfy the differential equations
\begin{equation}\label{Eq_Part_Greenfunction_DGL}
i\frac{\partial}{\partial t}G_j(t,x,z)=-\frac{\partial^2}{\partial x^2}G_j(t,x,z),\qquad j\in\{0,1,\textnormal{free}\}.
\end{equation}
\end{lem}

\begin{proof}
In order to verify \eqref{Eq_Part_Greenfunction_DGL} we compute the derivatives of the functions \eqref{Eq_Part_Greenfunction} explicitly. For $G_0$ we get
\begin{subequations}\label{Eq_Derivative_G0}
\begin{align}
\frac{\partial}{\partial t}G_0(t,x,z)&=\frac{i}{4t\sqrt{i\pi t}}\left(i-\frac{(|x|+z)^2}{2t}\right)e^{-\frac{(|x|+z)^2}{4it}}, \label{Eq_Derivative_G0_1} \\
\frac{\partial}{\partial x}G_0(t,x,z)&=-\sgn(x)\frac{|x|+z}{4it\sqrt{i\pi t}}e^{-\frac{(|x|+z)^2}{4it}}, \label{Eq_Derivative_G0_2} \\
\frac{\partial^2}{\partial x^2}G_0(t,x,z)&=\frac{1}{4t\sqrt{i\pi t}}\left(i-\frac{(|x|+z)^2}{2t}\right)e^{-\frac{(|x|+z)^2}{4it}}. \label{Eq_Derivative_G0_3}
\end{align}
\end{subequations}
For $G_1$ we use (\ref{Eq_Lambda_Derivative}) to obtain
\begin{subequations}\label{Eq_Derivative_G1}
\begin{align}
\frac{\partial}{\partial t}G_1(t,x,z)&=i\left(\omega^2\Lambda\left(\frac{|x|+z}{2\sqrt{it}}+\omega\sqrt{it}\right)+\frac{1}{it\sqrt{\pi}}\left(\frac{|x|+z}{2\sqrt{it}}-\omega\sqrt{it}\right)\right)e^{-\frac{(|x|+z)^2}{4it}}, \label{Eq_Derivative_G1_1} \\
\frac{\partial}{\partial x}G_1(t,x,z)&=\sgn(x)\left(\omega\Lambda\left(\frac{|x|+z}{2\sqrt{it}}+\omega\sqrt{it}\right)-\frac{1}{\sqrt{i\pi t}}\right)e^{-\frac{(|x|+z)^2}{4it}}, \label{Eq_Derivative_G1_2} \\
\frac{\partial^2}{\partial x^2}G_1(t,x,z)&=\left(\omega^2\Lambda\left(\frac{|x|+z}{2\sqrt{it}}+\omega\sqrt{it}\right)+\frac{1}{it\sqrt{\pi}}\left(\frac{|x|+z}{2\sqrt{it}}-\omega\sqrt{it}\right)\right)e^{-\frac{(|x|+z)^2}{4it}}. \label{Eq_Derivative_G1_3}
\end{align}
\end{subequations}
Finally, for $G_\text{free}$ we get, in a similar way as for $G_0$, the derivatives
\begin{subequations}\label{Eq_Derivative_Gfree}
\begin{align}
\frac{\partial}{\partial t}G_\text{free}(t,x,z)&=\frac{i}{4t\sqrt{i\pi t}}\left(i-\frac{(x-z)^2}{2t}\right)e^{-\frac{(x-z)^2}{4it}}, \label{Eq_Derivative_Gfree_1} \\
\frac{\partial}{\partial x}G_\text{free}(t,x,z)&=-\frac{x-z}{4it\sqrt{i\pi t}}e^{-\frac{(x-z)^2}{4it}}, \label{Eq_Derivative_Gfree_2} \\
\frac{\partial^2}{\partial x^2}G_\text{free}(t,x,z)&=\frac{1}{4t\sqrt{i\pi t}}\left(i-\frac{(x-z)^2}{2t}\right)e^{-\frac{(x-z)^2}{4it}}. \label{Eq_Derivative_Gfree_3}
\end{align}
\end{subequations}
In all three cases it is obvious that the differential equation \eqref{Eq_Part_Greenfunction_DGL} is satisfied.
\end{proof}

Next we will collect some elementary estimates of the functions $G_0$, $G_1$, and $G_\text{free}$, which will be needed throughout the paper.

\begin{lem}\label{cor_G_Estimate}
For every $t>0$, $x\in\mathbb{R}\setminus\{0\}$, and $z\in\mathbb{C}$ with $\Arg(z)\in[0,\frac{\pi}{2}]$ the following estimates for
the functions \eqref{Eq_Part_Greenfunction} hold:
\begin{subequations}\label{Eq_Estimate}
\begin{align}
\big|G_j(t,x,z)\big|&\leq c_j(t)\,e^{-\frac{\Im(z^2)}{4t}-\frac{|x|\Im(z)}{2t}},\qquad j\in\{0,1\}, \label{Eq_Estimate_2} \\
\big|G_\textnormal{free}(t,x,z)\big|&\leq c_\textnormal{free}(t)\,e^{-\frac{\Im(z^2)}{4t}+\frac{x\Im(z)}{2t}}, \label{Eq_Estimate_3}
\end{align}
\end{subequations}
where $c_0(t)=c_\textnormal{free}(t)=\frac{1}{2\sqrt{\pi t}}$ and $c_1(t)=\Lambda\big(\frac{\omega\sqrt{t}}{\sqrt{2}}\big)$. In particular,
the functions \eqref{Eq_Part_Greenfunction} satisfy the common estimate
\begin{equation}\label{Eq_Estimate_1}
\big|G_j(t,x,z)\big|\leq c_j(t)\,e^{-\frac{\Im(z^2)}{4t}+\frac{|x|\Im(z)}{2t}}, \qquad j\in\{0,1,\textnormal{free}\}.
\end{equation}
\end{lem}

\begin{proof}
The estimates \eqref{Eq_Estimate_2} and \eqref{Eq_Estimate_3} for $G_0$ and $G_\text{free}$ are obvious. For the estimate \eqref{Eq_Estimate_2}
of $G_1$ we use Lemma \ref{lem_Lambda_Properties} (iii) and (iv) to get
\begin{equation*}
\left|\Lambda\left(\frac{|x|+z}{2\sqrt{it}}+\omega\sqrt{it}\right)\right|\leq\Lambda\left(\frac{|x|+\Re(z)+\Im(z)}{2\sqrt{2t}}
+\frac{\omega\sqrt{t}}{\sqrt{2}}\right)\leq\Lambda\left(\frac{\omega\sqrt{t}}{\sqrt{2}}\right),
\end{equation*}
where the monotonicity of $\Lambda$ is applicable since $\Re(z),\Im(z)\geq 0$ due to $\Arg(z)\in[0,\frac{\pi}{2}]$. Finally,
the estimate \eqref{Eq_Estimate_1} follows immediately from \eqref{Eq_Estimate} by further estimating the exponents.
\end{proof}

Now we turn to our main objective in this section and introduce the Green's function
\begin{equation}\label{Eq_Greenfunction_decomposition}
\begin{split}
G(t,x,y)&=\mu_+^{(x,y)}G_1(t,x,|y|;\omega_+)+\mu_-^{(x,y)}G_1(t,x,|y|;\omega_-) \\
&\qquad +\mu_0^{(x,y)}G_0(t,x,|y|)+G_\text{free}(t,x,y),\qquad t>0,\,x,y\in\mathbb{R}\setminus\{0\},
\end{split}
\end{equation}
which is expressed in terms of the functions \eqref{Eq_Part_Greenfunction} and we have added the additional argument
$\omega_\pm$ in $G_1$ to emphasize the dependence of the parameter $\omega$ in \eqref{Eq_Part_Greenfunction1}.
The function in \eqref{Eq_Greenfunction_decomposition} coincides with the Green's function \eqref{Eq_Greenfunction} mentioned in
the Introduction.
We prove in Theorem~\ref{satz_Green_function} below that for a proper choice of coefficients $\mu_\pm$ and $\mu_0$
the function \eqref{Eq_Greenfunction_decomposition} satisfies the differential equation \eqref{Eq_Schroedinger_1}
as well as the jump condition \eqref{Eq_Schroedinger_2} for a fixed unitary matrix $J$. The connection to the initial value \eqref{Eq_Schroedinger_3} is postponed to
Lemma \ref{lem_Psi} and Theorem~\ref{satz_Wave_function} in Section~\ref{sec_Solution_of_the_Schroedinger_equation}, where
the precise meaning of the integral \eqref{Eq_Green_function_integral_sec2} is clarified first.

Next we provide the coefficients $\omega_\pm$ and the piecewise constant functions $\mu_\pm$ and $\mu_0$ explicitly in
terms of the unitary $2\times 2$-matrix $J$ in \eqref{Eq_Schroedinger_2}. Note that
\begin{equation}\label{Eq_General_unitary_matrix}
J=e^{i\phi}\mmatrix{\alpha}{-\bar\beta}{\beta}{\bar\alpha},
\end{equation}
with parameters $\phi\in[0,\pi)$ and $\alpha,\beta\in\mathbb{C}$ satisfying $|\alpha|^2+|\beta|^2=1$.
It is convenient to use
\begin{subequations}\label{Eq_eta}
\begin{align}
\eta^{(x,y)}&:=\frac{1}{\sqrt{1-\Re(\alpha)^2}}\left\{\begin{array}{ll} -\Im(\alpha), & \text{if }x,y>0, \\ -i\bar\beta, & \text{if }x>0,\,y<0, \\ i\beta, & \text{if }x<0,\,y>0, \\ \Im(\alpha), & \text{if }x,y<0, \end{array}\right. \qquad\text{if }|\Re(\alpha)|\neq 1, \label{Eq_eta1} \\
\eta^{(x,y)}&:= 0,\hspace{7.95cm}\text{if }|\Re(\alpha)|=1, \label{Eq_eta2}
\end{align}
\end{subequations}
the step function
$$\Theta(x)=\left\{\begin{array}{ll} 1, & \text{if }x>0, \\ 0, & \text{if }x<0, \end{array}\right.$$
and to distinguish the following three cases.

\medskip
\noindent
\textbf{Case I}: If $\Re(\alpha)\neq-\cos(\phi)$, then
\begin{equation*}
\omega_\pm=\frac{-\sin(\phi)\pm\sqrt{1-\Re(\alpha)^2}}{\cos(\phi)+\Re(\alpha)},\quad\mu_\pm^{(x,y)}=-\frac{\omega_\pm}{2}\big(\Theta(xy)\pm\eta^{(x,y)}\big),\quad\mu_0^{(x,y)}=\sgn(xy).
\end{equation*}

\noindent
\textbf{Case II}: If $\Re(\alpha)=-\cos(\phi)\neq-1$, then $\omega_-=\mu_-^{(x,y)}=0$ and
\begin{equation*}
\omega_+=\cot(\phi),\quad\mu_+^{(x,y)}=-\frac{\omega_+}{2}\big(\Theta(xy)+\eta^{(x,y)}\big),\quad\mu_0^{(x,y)}=\eta^{(x,y)}-\Theta(-xy).
\end{equation*}

\noindent
\textbf{Case III}: If $\Re(\alpha)=-\cos(\phi)=-1$, then $\omega_\pm=\mu_\pm^{(x,y)}=0$ and $\mu_0^{(x,y)}=-1$.

\medskip

These three cases correspond to the rank of the matrix $I+J$ on the right hand side of the jump conditions \eqref{Eq_Schroedinger_2} or
\eqref{Eq_Jump_condition_Greenfunction}.
More precisely, in Case~I we have $\rank(I+J)=2$, in Case~II we have $\rank(I+J)=1$, and finally, in Case~III we have $\rank(I+J)=0$.

\begin{satz}\label{satz_Green_function}
For every fixed $y\in\mathbb{R}\setminus\{0\}$ the Green's function \eqref{Eq_Greenfunction_decomposition} satisfies the differential equation
\begin{equation}\label{Eq_DGL_Greenfunction}
i\frac{\partial}{\partial t}G(t,x,y)=-\frac{\partial^2}{\partial x^2}G(t,x,y),\qquad t>0,\,x\in\mathbb{R}\setminus\{0\},
\end{equation}
as well as the jump condition
\begin{equation}\label{Eq_Jump_condition_Greenfunction}
(I-J)\vvect{G(t,0^+,y)}{G(t,0^-,y)}=i(I+J)\vvect{\frac{\partial}{\partial x}G(t,0^+,y)}{-\frac{\partial}{\partial x}G(t,0^-,y)},\qquad t>0.
\end{equation}
\end{satz}

\begin{proof}
Note first that the coefficients $\mu_\pm^{(x,y)}$ and $\mu_0^{(x,y)}$ in the representation \eqref{Eq_Greenfunction_decomposition} of the function $G$
only depend on the signs
of $x$ and $y$. In particular, the coefficients are constant on the half lines $x>0$ and $x<0$, and hence it
follows from Lemma \ref{lem_Part_Greenfunction_DGL} that the function $G$ in \eqref{Eq_Greenfunction_decomposition} is a solution of the differential equation \eqref{Eq_DGL_Greenfunction}.

\medskip

In the following we will verify that the jump condition \eqref{Eq_Jump_condition_Greenfunction} is satisfied.
Using \eqref{Eq_Derivative_G0_2}, \eqref{Eq_Derivative_G1_2}, and \eqref{Eq_Derivative_Gfree_2} we find that the spatial derivative of the function $G$
is given by
\begin{align*}
\frac{\partial}{\partial x}G(t,x,y)&=\mu_+^{(x,y)}\sgn(x)\left(\omega_+\Lambda\left(\frac{|x|+|y|}{2\sqrt{it}}+\omega_+\sqrt{it}\right)-\frac{1}{\sqrt{i\pi t}}\right)e^{-\frac{(|x|+|y|)^2}{4it}} \\
&\quad+\mu_-^{(x,y)}\sgn(x)\left(\omega_-\Lambda\left(\frac{|x|+|y|}{2\sqrt{it}}+\omega_-\sqrt{it}\right)-\frac{1}{\sqrt{i\pi t}}\right)e^{-\frac{(|x|+|y|)^2}{4it}} \\
&\quad-\frac{1}{4it\sqrt{i\pi t}}\left(\mu_0^{(x,y)}\sgn(x)\big(|x|+|y|\big)e^{-\frac{(|x|+|y|)^2}{4it}}+(x-y)e^{-\frac{(x-y)^2}{4it}}\right).
\end{align*}
For the jump condition \eqref{Eq_Jump_condition_Greenfunction} we have to evaluate $G$ and $\frac{\partial}{\partial x}G$ at $x=0^\pm$.
As in \eqref{Eq_Jump_condition_Greenfunction} this will be done in a vector form, where the first entry is the limit $x=0^+$ and the second entry the limit $x=0^-$.
We have
\begin{align*}
\vvect{G(t,0^+,y)}{G(t,0^-,y)}&=\left(\vvect{\mu_+^{(0^+,y)}}{\mu_+^{(0^-,y)}}\Lambda\left(\frac{|y|}{2\sqrt{it}}+\omega_+\sqrt{it}\right)+\vvect{\mu_-^{(0^+,y)}}{\mu_-^{(0^-,y)}}\Lambda\left(\frac{|y|}{2\sqrt{it}}+\omega_-\sqrt{it}\right)\right. \\
&\quad\left.+\frac{1}{2\sqrt{i\pi t}}\vvect{\mu_0^{(0^+,y)}+1}{\mu_0^{(0^-,y)}+1}\right)e^{-\frac{y^2}{4it}}, \\
\vvect{\frac{\partial}{\partial x}G(t,0^+,y)}{-\frac{\partial}{\partial x}G(t,0^-,y)}
&=\left(\vvect{\mu_+^{(0^+,y)}}{\mu_+^{(0^-,y)}}\omega_+\Lambda\left(\frac{|y|}{2\sqrt{it}}+\omega_+\sqrt{it}\right)\right.\\
&\quad\left.+\vvect{\mu_-^{(0^+,y)}}{\mu_-^{(0^-,y)}}\omega_-\Lambda\left(\frac{|y|}{2\sqrt{it}}+\omega_-\sqrt{it}\right)\right. \\
&\quad\left.-\frac{1}{\sqrt{i\pi t}}\vvect{\mu_+^{(0^+,y)}+\mu_-^{(0^+,y)}}{\mu_+^{(0^-,y)}+\mu_-^{(0^-,y)}}-\frac{|y|}{4it\sqrt{i\pi t}}\vvect{\mu_0^{(0^+,y)}-\sgn(y)}{\mu_0^{(0^-,y)}+\sgn(y)}\right)e^{-\frac{y^2}{4it}},
\end{align*}
and since \eqref{Eq_Jump_condition_Greenfunction} has to be satisfied for all $y\in\mathbb{R}\setminus\{0\}$ it suffices to compare and match the coefficients
corresponding to the terms
\begin{equation*}
 \Lambda\left(\frac{|y|}{2\sqrt{it}}+\omega_\pm\sqrt{it}\right),\quad\frac{1}{2\sqrt{i\pi t}},\quad\text{and}\quad\frac{|y|}{4it\sqrt{i\pi t}},
\end{equation*}
which leads to the following four equations
\begin{subequations}
\begin{align*}
(\textrm{A}_\pm):&\;\;(I-J)\vvect{\mu_\pm^{(0^+,y)}}{\mu_\pm^{(0^-,y)}}=i\omega_\pm(I+J)\vvect{\mu_\pm^{(0^+,y)}}{\mu_\pm^{(0^-,y)}},  \\
(\textrm{B})\,\,:&\;\;(I-J)\vvect{\mu_0^{(0^+,y)}+1}{\mu_0^{(0^-,y)}+1}=-2i(I+J)\vvect{\mu_+^{(0^+,y)}+\mu_-^{(0^+,y)}}{\mu_+^{(0^-,y)}+\mu_-^{(0^-,y)}}, \\
(\textrm{C})\,\,:&\;\;\vvect{0}{0}=(I+J)\vvect{\mu_0^{(0^+,y)}-\sgn(y)}{\mu_0^{(0^-,y)}+\sgn(y)}.
\end{align*}
\end{subequations}
Since the variable $y$ only appears as $\sgn(y)$ each equation splits up in one for $y>0$ and one for $y<0$.
We will consider this by writing $(\textrm{A}_\pm)$, $(\textrm{B})$, and $(\textrm{C})$ as matrix equations, where the first column is for $y>0$ and the second column for $y<0$. For a shorter notation we will use the matrices
\begin{equation}\label{Eq_Matrices}
\mathbbm{1}:=\mmatrix{1}{1}{1}{1},\quad N:=\mmatrix{\eta^{(0^+,0^+)}}{\eta^{(0^+,0^-)}}{\eta^{(0^-,0^+)}}{\eta^{(0^-,0^-)}},\quad M_j:=\mmatrix{\mu_j^{(0^+,0^+)}}{\mu_j^{(0^+,0^-)}}{\mu_j^{(0^-,0^+)}}{\mu_j^{(0^-,0^-)}},
\end{equation}
where $j\in\{0,\pm\}$. Note that the matrix $N$ satisfies the identity
\begin{equation}\label{Eq_Eta_identity}
\sqrt{1-\Re(\alpha)^2}\,N=\mmatrix{-\Im(\alpha)}{-i\bar\beta}{i\beta}{\Im(\alpha)}
\end{equation}
by \eqref{Eq_eta1} for $|\Re(\alpha)|\neq 1$ and also for $|\Re(\alpha)|=1$, since then $\Im(\alpha)=\beta=0$ due to $|\alpha|^2+|\beta|^2=1$.
From \eqref{Eq_Eta_identity} and $|\alpha|^2+|\beta|^2=1$ it immediately follows that
\begin{equation*}
N^2=\frac{1}{1-\Re(\alpha)^2}\mmatrix{-\Im(\alpha)}{-i\bar\beta}{i\beta}{\Im(\alpha)}^2=\frac{\Im(\alpha)^2+|\beta|^2}{1-\Re(\alpha)^2}\;I=I,\qquad\text{if }|\Re(\alpha)|\neq 1,
\end{equation*}
and, consequently,
\begin{equation}\label{Eq_etasquare}
(N+I)(N-I)=N^2-N+N-I=N^2-I=0,\qquad\text{if }|\Re(\alpha)|\neq 1,
\end{equation}
to which we will refer throughout the proof. With the help of the matrices \eqref{Eq_Matrices} we now rewrite the equations $(\textrm{A}_\pm)$, $(\textrm{B})$,
and $(\textrm{C})$ above in the matrix form
\begin{align*}
(\textrm{A}_\pm):&\;\;(I-J)M_\pm=i\omega_\pm(I+J)M_\pm, \\
(\textrm{B})\,\,:&\;\;(I-J)(M_0+\mathbbm{1})=-2i(I+J)(M_++M_-), \\
(\textrm{C})\,\,:&\;\;0=(I+J)(M_0+\mathbbm{1}-2I).
\end{align*}
Plugging in the matrix $J$ from \eqref{Eq_General_unitary_matrix} and multiplying by $e^{-i\phi}$ these equations turn into
\begin{subequations}\label{Eq_ABC}
\begin{align}
(\textrm{A}_\pm):&\;\;\mmatrix{e^{-i\phi}-\alpha}{\bar\beta}{-\beta}{e^{-i\phi}-\bar\alpha}M_\pm=i\omega_\pm\mmatrix{e^{-i\phi}+\alpha}{-\bar\beta}{\beta}{e^{-i\phi}+\bar\alpha}M_\pm, \label{Eq_A} \\
(\textrm{B})\,\,\,:&\;\;\mmatrix{e^{-i\phi}-\alpha}{\bar\beta}{-\beta}{e^{-i\phi}-\bar\alpha}(M_0+\mathbbm{1})=-2i\mmatrix{e^{-i\phi}+\alpha}{-\bar\beta}{\beta}{e^{-i\phi}+\bar\alpha}(M_++M_-), \label{Eq_B} \\
(\textrm{C})\,\,\,:&\;\;0=\mmatrix{e^{-i\phi}+\alpha}{-\bar\beta}{\beta}{e^{-i\phi}+\bar\alpha}(M_0+\mathbbm{1}-2I). \label{Eq_C}
\end{align}
\end{subequations}
In the following we will discuss the three cases above Theorem~\ref{satz_Green_function} separately and verify that in each case with the proper choice of
the coefficients $\omega_\pm$ and $\mu_\pm,\mu_0$ the
equations $(\textrm{A}_\pm)$, $(\textrm{B})$,
and $(\textrm{C})$ are satisfied, that is, the jump condition \eqref{Eq_Jump_condition_Greenfunction} holds.

\medskip
\noindent
\textbf{Case I}. Observe first that the equation \eqref{Eq_C} is satisfied since $\mu_0^{(x,y)}=\sgn(xy)$ in this case, and hence we conclude
$M_0=2I-\mathbbm{1}$. Next we use $|\alpha|^2+|\beta|^2=1$ to compute
\begin{equation*}
\det\mmatrix{e^{-i\phi}+\alpha}{-\bar\beta}{\beta}{e^{-i\phi}+\bar\alpha}=2e^{-i\phi}\big(\cos(\phi)+\Re(\alpha)\big)\neq 0,
\end{equation*}
where we also used the assumption $\Re(\alpha)\neq-\cos(\phi)$ in Case I.
It follows that the matrix on the right hand side of $(\textrm{A}_\pm)$ and $(\textrm{B})$ is invertible with the inverse
\begin{equation*}
\mmatrix{e^{-i\phi}+\alpha}{-\bar\beta}{\beta}{e^{-i\phi}+\bar\alpha}^{-1}=\frac{e^{i\phi}}{2(\cos(\phi)+\Re(\alpha))}
\mmatrix{e^{-i\phi}+\bar\alpha}{\bar\beta}{-\beta}{e^{-i\phi}+\alpha},
\end{equation*}
and this leads to
\begin{align*}
&\mmatrix{e^{-i\phi}+\alpha}{-\bar\beta}{\beta}{e^{-i\phi}+\bar\alpha}^{-1}\mmatrix{e^{-i\phi}-\alpha}{\bar\beta}{-\beta}{e^{-i\phi}-\bar\alpha} \\
&\hspace{2cm}=\frac{-i}{\cos(\phi)+\Re(\alpha)}\mmatrix{\sin(\phi)+\Im(\alpha)}{i\bar\beta}{-i\beta}{\sin(\phi)-\Im(\alpha)} \\
&\hspace{2cm}=\frac{-i}{\cos(\phi)+\Re(\alpha)}\bigl(\sin(\phi)\,I-\sqrt{1-\Re(\alpha)^2}\,N\bigr),
\end{align*}
where in the last line we used the identity \eqref{Eq_Eta_identity}. Hence the equations \eqref{Eq_A} and \eqref{Eq_B} turn into
\begin{align*}
(\textrm{A}_\pm):&\;\;\frac{\sin(\phi)\,I-\sqrt{1-\Re(\alpha)^2}\,N}{\cos(\phi)+\Re(\alpha)}M_\pm=-\omega_\pm M_\pm, \\
(\textrm{B})\,\,:&\;\;\frac{\sin(\phi)\,I-\sqrt{1-\Re(\alpha)^2}\,N}{\cos(\phi)+\Re(\alpha)}(M_0+\mathbbm{1})=2(M_++M_-).
\end{align*}
Using the explicit form $\omega_\pm=\frac{-\sin(\phi)\pm\sqrt{1-\Re(\alpha)^2}}{\cos(\phi)+\Re(\alpha)}$ in $(\textrm{A}_\pm)$ and $M_0=2I-\mathbbm{1}$ in $(\textrm{B})$
these equations reduce to
\begin{align*}
(\textrm{A}_\pm):&\;\;\sqrt{1-\Re(\alpha)^2}(N\mp I)M_\pm=0, \\
(\textrm{B})\,\,:&\;\;\frac{\sin(\phi)\,I-\sqrt{1-\Re(\alpha)^2}\,N}{\cos(\phi)+\Re(\alpha)}=M_++M_-.
\end{align*}
Since we treat Case I we have $\mu_\pm^{(x,y)}=-\frac{\omega_\pm}{2}\big(\Theta(xy)\pm\eta^{(x,y)}\big)$ and from that we conclude
\begin{equation}\label{mmm}
M_\pm=-\frac{\omega_\pm}{2}(I\pm N).
\end{equation}
In particular, this yields
\begin{equation*}
M_++M_-=-\frac{(\omega_++\omega_-)I+(\omega_+-\omega_-)N}{2}=\frac{\sin(\phi)I-\sqrt{1-\Re(\alpha)^2}\,N}{\cos(\phi)+\Re(\alpha)},
\end{equation*}
which shows that equation $(\textrm{B})$ is valid.
It remains to check $(\textrm{A}_\pm)$. These equations are obviously valid if $|\Re(\alpha)|=1$ and if
$|\Re(\alpha)|\neq 1$ they follow from the identities \eqref{Eq_etasquare} and \eqref{mmm}.

\medskip
\noindent
\textbf{Case II}.  Here we assume $\Re(\alpha)=-\cos(\phi)\neq -1$, which implies, in particular, $\phi\neq 0$ and consequently
$\sin(\phi)\neq 0$. The matrices in the equations $(\textrm{A}_\pm)$, $(\textrm{B})$,
and $(\textrm{C})$ in \eqref{Eq_ABC} now have the form
\begin{align*}
\mmatrix{e^{-i\phi}-\alpha}{\bar\beta}{-\beta}{e^{-i\phi}-\bar\alpha}&=\big(2\cos(\phi)-i\sin(\phi)\big)I+i\mmatrix{-\Im(\alpha)}{-i\bar\beta}{i\beta}{\Im(\alpha)} \\
&=-i\sin(\phi)\big((2i\cot(\phi)+1)I-N\big), \\
\mmatrix{e^{-i\phi}+\alpha}{-\bar\beta}{\beta}{e^{-i\phi}+\bar\alpha}&=-i\sin(\phi)I-i\mmatrix{-\Im(\alpha)}{-i\bar\beta}{i\beta}{\Im(\alpha)} \\
&=-i\sin(\phi)(I+N),
\end{align*}
where in both cases we used \eqref{Eq_Eta_identity} and $\sqrt{1-\Re(\alpha)^2}=\sin(\phi)$, because $\Re(\alpha)=-\cos(\phi)$. Using this in \eqref{Eq_ABC} leads to
\begin{align*}
(\textrm{A}_\pm):&\;\;\big((2i\cot(\phi)+1)I-N\big)M_\pm=i\omega_\pm(I+N)M_\pm, \\
(\textrm{B})\,\,:&\;\;\big((2i\cot(\phi)+1)I-N\big)(M_0+\mathbbm{1})=-2i(I+N)(M_++M_-), \\
(\textrm{C})\,\,:&\;\;0=(I+N)(M_0+\mathbbm{1}-2I).
\end{align*}
Since in Case II we have $\mu_-^{(x,y)}=0$, that is, $M_-=0$, the equation $(\textrm{A}_-)$ is trivially satisfied. Furthermore,
with our choice $\omega_+=\cot(\phi)$ the equation $(\textrm{A}_+)$ reduces to
\begin{equation*}
(\textrm{A}_+):\;\;(i\cot(\phi)+1)(I-N)M_+=0.
\end{equation*}
By our choice of $\mu_+^{(x,y)}$ we have $M_+=-\frac{\omega_+}{2}(I+N)$ as in the previous case (cf. \eqref{mmm}) and hence we conclude together with
\eqref{Eq_etasquare} that equation $(\textrm{A}_+)$ is valid; note that we can apply \eqref{Eq_etasquare} since $\Re(\alpha)\neq -1$ by the assumption in Case II and also
$\Re(\alpha)=-\cos(\phi)\neq 1$ as $\phi\in[0,\pi)$.
Next, we observe that also equation $(\textrm{C})$ holds by \eqref{Eq_etasquare} and $\mu_0^{(x,y)}=\eta^{(x,y)}-\Theta(-xy)$, which gives $M_0=N-\mathbbm{1}+I$.
In order to check $(\textrm{B})$, we plug in the above values for $M_0$ and $M_\pm$ and obtain
\begin{equation*}
(\textrm{B}):\;\;\big(1+i\cot(\phi)\big)(I-N)(N+I)=0,
\end{equation*}
which holds by \eqref{Eq_etasquare}.

\medskip
\noindent
\textbf{Case III}.
Here we assume $\Re(\alpha)=-\cos(\phi)=-1$ and hence $\Im(\alpha)=\beta=\phi=0$ follows from the condition $|\alpha|^2+|\beta|^2=1$.
Therefore, the equations $(\textrm{A}_\pm)$, $(\textrm{B})$,
and $(\textrm{C})$ in \eqref{Eq_ABC} have the particularly simple form
\begin{align*}
(\textrm{A}_\pm):&\;\;2M_\pm=0, \\
(\textrm{B})\,:&\;\;M_0+\mathbbm{1}=0, \\
(\textrm{C})\,:&\;\;0=0,
\end{align*}
and are all obviously satisfied by the definition of the coefficients in Case III.
\end{proof}

\section{Special cases of generalized point interactions and their Green's functions}\label{Special cases }

In this section we consider some particular generalized point interactions and derive the explicit form of the Green's function in these situations.
As an almost trivial case we start with the free particle in Example~\ref{bsp_Free}, discuss the well-known $\delta$ and $\delta'$-interactions afterwards
in Example~\ref{bsp_Delta} and Example~\ref{bsp_Deltaprime}, respectively, and in Examples~\ref{bsp_Dirichlet}--\ref{bsp_Robin} we treat decoupled systems with Dirichlet,
Neumann, and Robin boundary conditions at the origin.
In each of the examples we first provide the corresponding matrix $J$ for the interface conditions \eqref{Eq_Schroedinger_2} with parameters $\phi,\alpha,\beta$
as in \eqref{Eq_General_unitary_matrix},
then we determine which of the Cases I--III above Theorem~\ref{satz_Green_function} appears, and finally we compute the coefficients in the
Green's function \eqref{Eq_Greenfunction} or \eqref{Eq_Greenfunction_decomposition}.
The special Green's functions in this section are known from the mathematical and physical literature.

\begin{bsp}[Free particle]\label{bsp_Free}
The wave function corresponding to
a free particle is continuous with continuous first derivative and hence at the point $x=0$ we have
\begin{equation*}
\Psi(t,0^-)=\Psi(t,0^+)\qquad\text{and}\qquad\frac{\partial}{\partial x}\Psi(t,0^-)=\frac{\partial}{\partial x}\Psi(t,0^+),\qquad t>0.
\end{equation*}
These continuity conditions are described in \eqref{Eq_Schroedinger_2} if we consider the matrix
\begin{equation*}
J=\mmatrix{0}{1}{1}{0}.
\end{equation*}
This matrix is of the form \eqref{Eq_General_unitary_matrix} with $\alpha=0$, $\beta=-i$, and $\phi=\frac{\pi}{2}$.
In this situation the coefficient $\eta^{(x,y)}$ in \eqref{Eq_eta1} is
\begin{equation*}
\eta^{(x,y)}=\left\{\begin{array}{ll} 0, & \textnormal{if }x,y>0, \\ 1, & \textnormal{if }x>0,\,y<0, \\ 1, & \textnormal{if }x<0,\,y>0, \\ 0, & \textnormal{if }x,y<0,
\end{array}\right.=\Theta(-xy).
\end{equation*}
Since we are in Case II the coefficients of the corresponding Green function in \eqref{Eq_Greenfunction} have the explicit form
\begin{align*}
\omega_-=0,\hspace{2.55cm}\mu_-^{(x,y)}&=0, \\
\omega_+=\cot\Big(\frac{\pi}{2}\Big)=0,\qquad\mu_+^{(x,y)}&=-\frac{\omega_+}{2}\big(\Theta(xy)+\eta^{(x,y)}\big)=0, \\
\mu_0^{(x,y)}&=\eta^{(x,y)}-\Theta(-xy)=0.
\end{align*}
Therefore, the Green's function of the free particle is given by
\begin{equation*}
G(t,x,y)=\frac{1}{2\sqrt{i\pi t}}e^{-\frac{(x-y)^2}{4it}}.
\end{equation*}
\end{bsp}

In the next example we treat the classical $\delta$-point interaction located at the origin. Such singular potentials
were studied intensively in the mathematical and physical literature; we refer the interested reader to the standard monograph \cite{ALB} for
a detailed treatment and further references. The particular Green's function that appears below can also be found (sometimes in a slightly different form)
in the papers \cite{C09,GaSc86,Ma89}.

\begin{bsp}[$\delta$-potential]\label{bsp_Delta}
We consider the standard $\delta$-interaction of strength $2c\in\mathbb{R}\setminus\{0\}$ located at the point $x=0$.
This situation is described by the formal Schr\"odinger equation
\begin{equation*}
i\frac{\partial}{\partial t}\Psi(t,x)=\Big(-\frac{\partial^2}{\partial x^2}+2c\delta(x)\Big)\Psi(t,x),\qquad t>0,\,x\in\mathbb{R},
\end{equation*}
and is made mathematically rigorous in the form
\begin{subequations}\label{Eq_Schroedinger_delta}
\begin{align}
i\frac{\partial}{\partial t}\Psi(t,x)&=-\frac{\partial^2}{\partial x^2}\Psi(t,x),\qquad t>0,\,x\in\mathbb{R}\setminus\{0\}, \notag \\
\Psi(t,0^+)&=\Psi(t,0^-),\hspace{1.6cm} t>0, \label{Eq_Schroedinger_delta1} \\
\frac{\partial}{\partial x}\Psi(t,0^+)-\frac{\partial}{\partial x}\Psi(t,0^-)&=2c\,\Psi(t,0^\pm),\hspace{1.2cm} t>0. \label{Eq_Schroedinger_delta2}
\end{align}
\end{subequations}
The jump condition \eqref{Eq_Schroedinger_delta1}--\eqref{Eq_Schroedinger_delta2} is realized in
\eqref{Eq_Schroedinger_2} by using the matrix
\begin{equation*}
J=\frac{1}{i-c}\mmatrix{c}{i}{i}{c}.
\end{equation*}
In fact, with this choice of $J$ and multiplication by $(c-i)$ the condition \eqref{Eq_Schroedinger_2} reads as
\begin{equation*}
\mmatrix{2c-i}{i}{i}{2c-i}\vvect{\Psi(t,0^+)}{\Psi(t,0^-)}=\mmatrix{1}{1}{1}{1}\vvect{\frac{\partial}{\partial x}\Psi(t,0^+)}{-\frac{\partial}{\partial x}\Psi(t,0^-)},
\end{equation*}
or, more explicitely, we have the two equations
\begin{equation*}
 \begin{split}
  (2c-i)\Psi(t,0^+) + i\Psi(t,0^-)  &= \frac{\partial}{\partial x}\Psi(t,0^+) - \frac{\partial}{\partial x}\Psi(t,0^-), \\
  i\Psi(t,0^+) +  (2c-i)\Psi(t,0^-) &= \frac{\partial}{\partial x}\Psi(t,0^+) - \frac{\partial}{\partial x}\Psi(t,0^-).
 \end{split}
\end{equation*}
By subtracting these equations from each other we first conclude \eqref{Eq_Schroedinger_delta1} and adding the equations leads to \eqref{Eq_Schroedinger_delta2}.
In order to write the matrix $J$ in the form \eqref{Eq_General_unitary_matrix} we choose $\phi\in(0,\pi)$ such that $\cot(\phi)=c$.
Next we set $\alpha=-\cos(\phi)$ and $\beta=-i\sin(\phi)$.  It follows, in particular, that
\begin{equation*}
\cos(\phi)=\frac{c}{\sqrt{1+c^2}}\qquad\text{and}\qquad\sin(\phi)=\frac{1}{\sqrt{1+c^2}},
\end{equation*}
and therefore
\begin{equation*}
e^{i\phi}\mmatrix{\alpha}{-\bar\beta}{\beta}{\bar\alpha}=\frac{1}{i-c}\mmatrix{c}{i}{i}{c}=J.
\end{equation*}
Plugging these values in \eqref{Eq_eta1} gives
\begin{equation*}
\eta^{(x,y)}=\left\{\begin{array}{ll} 0, & \text{if }x,y>0, \\ 1, & \text{if }x>0,\,y<0, \\ 1, & \text{if }x<0,\,y>0, \\ 0, & \text{if }x,y<0, \end{array}\right.=\Theta(-xy),
\end{equation*}
and since we are in Case II the coefficients of the Green's function are
\begin{align*}
\omega_-=0,\hspace{2.3cm}\mu_-^{(x,y)}&=0, \\
\omega_+=\cot(\phi)=c,\qquad\mu_+^{(x,y)}&=-\frac{c}{2}\big(\Theta(xy)+\Theta(-xy)\big)=-\frac{c}{2}, \\
\mu_0^{(x,y)}&=\Theta(-xy)-\Theta(-xy)=0.
\end{align*}
With these quantities we conclude from \eqref{Eq_Greenfunction} that the Green's function of the $\delta$-potential is given by
\begin{equation*}
G(t,x,y)=-\frac{c}{2}\Lambda\left(\frac{|x|+|y|}{2\sqrt{it}}+c\sqrt{it}\right)e^{-\frac{(|x|+|y|)^2}{4it}}+\frac{1}{2\sqrt{i\pi t}}e^{-\frac{(x-y)^2}{4it}}.
\end{equation*}
\end{bsp}

The $\delta'$-interaction in the next example is another popular singular potential that appears in various situations. 

\begin{bsp}[$\delta'$-potential]\label{bsp_Deltaprime}
Now consider the $\delta'$-interaction of strength $\frac{2}{c}\in\mathbb{R}\setminus\{0\}$ located at the point $x=0$. Formally one then deals with
the Schr\"odinger equation
\begin{equation*}
i\frac{\partial}{\partial t}\Psi(t,x)=\Big(-\frac{\partial^2}{\partial x^2}+\frac{2}{c}\delta'(x)\Big)\Psi(t,x),\qquad t>0,\,x\in\mathbb{R},
\end{equation*}
which in a mathematically rigorous form reads as
\begin{align*}
i\frac{\partial}{\partial t}\Psi(t,x)&=-\frac{\partial^2}{\partial x^2}\Psi(t,x),\qquad t>0,\,x\in\mathbb{R}\setminus\{0\}, \\
\frac{\partial}{\partial x}\Psi(t,0^+)&=\frac{\partial}{\partial x}\Psi(t,0^-), \hspace{1.05cm} t>0, \\
\Psi(t,0^+)-\Psi(t,0^-)&=\frac{2}{c}\frac{\partial}{\partial x}\Psi(t,0),\hspace{1.05cm} t>0.
\end{align*}
One verifies in a similar way as in in the previous example that the jump conditions are realized in \eqref{Eq_Schroedinger_2} by using the matrix
\begin{equation*}
J=\frac{1}{i-c}\mmatrix{i}{-c}{-c}{i}.
\end{equation*}
This matrix is of the form \eqref{Eq_General_unitary_matrix} if we choose $\phi\in(0,\pi)\setminus\{\frac{\pi}{2}\}$ such that $\tan(\phi)=-c$
and set $\alpha=\cos(\phi)$ and $\beta=-i\sin(\phi)$. The coefficient $\eta^{(x,y)}$ in \eqref{Eq_eta1} then becomes
\begin{equation*}
\eta^{(x,y)}=\left\{\begin{array}{ll} 0, & \text{if }x,y>0, \\ 1, & \text{if }x>0,\,y<0, \\ 1, & \text{if }x<0,\,y>0, \\ 0, & \text{if }x,y<0, \end{array}\right.=\Theta(-xy),
\end{equation*}
and since we are in Case I the coefficients of the Green's function are
\begin{align*}
\omega_-=-\tan(\phi)=c,\qquad\mu_-^{(x,y)}&=-\frac{c}{2}\big(\Theta(xy)-\Theta(-xy)\big)=-\frac{c\,\sgn(xy)}{2}, \\
\omega_+=0,\hspace{2.6cm}\mu_+^{(x,y)}&=0, \\
\mu_0^{(x,y)}&=\sgn(xy).
\end{align*}
It follows that the Green's function of the $\delta'$-potential is given by
\begin{align*}
G(t,x,y)=&-\frac{c\,\sgn(xy)}{2}\Lambda\left(\frac{|x|+|y|}{2\sqrt{it}}+c\sqrt{it}\right)e^{-\frac{(|x|+|y|)^2}{4it}} \\
&+\frac{1}{2\sqrt{i\pi t}}\left(\sgn(xy)e^{-\frac{(|x|+|y|)^2}{4it}}+e^{-\frac{(x-y)^2}{4it}}\right).
\end{align*}
\end{bsp}

Now we turn to generalized point interactions that lead to decoupled systems. In the following examples we discuss
Dirichlet, Neumann, and Robin boundary conditions at the origin.

\begin{bsp}[Dirichlet boundary conditions]\label{bsp_Dirichlet}
We consider the free Schrödinger equation on the two half lines $\mathbb{R}\setminus\{0\}$ with Dirichlet boundary conditions
\begin{align*}
i\frac{\partial}{\partial t}\Psi(t,x)&=-\frac{\partial^2}{\partial x^2}\Psi(t,x),\qquad t>0,\,x\in\mathbb{R}\setminus\{0\}, \\
\Psi(t,0^+)&=\Psi(t,0^-)=0,\hspace{0.85cm} t>0.
\end{align*}
These boundary conditions are realized in \eqref{Eq_Schroedinger_2} by using the matrix
\begin{equation}\label{jdir}
J=\mmatrix{-1}{0}{0}{-1},
\end{equation}
that is, we have $\phi=0$, $\alpha=-1$, and $\beta=0$ in \eqref{Eq_General_unitary_matrix}, and hence Case III applies.
The coefficients of the Green's function are given by
\begin{equation*}
\omega_\pm=0,\qquad\mu_\pm^{(x,y)}=0,\quad\text{and}\quad\mu_0^{(x,y)}=-1,
\end{equation*}
and lead to
\begin{equation}\label{Eq_Greenfunction_Dirichlet}
G(t,x,y)=\frac{1}{2\sqrt{i\pi t}}\left(e^{-\frac{(x-y)^2}{4it}}-e^{-\frac{(|x|+|y|)^2}{4it}}\right).
\end{equation}
\end{bsp}

\begin{bsp}[Neumann boundary conditions]\label{bsp_Neumann}
We consider the free Schr\"odinger equation on the two half lines $\mathbb{R}\setminus\{0\}$ with Neumann boundary conditions
\begin{align*}
i\frac{\partial}{\partial t}\Psi(t,x)&=-\frac{\partial^2}{\partial x^2}\Psi(t,x),\hspace{1.25cm} t>0,\,x\in\mathbb{R}\setminus\{0\}, \\
\frac{\partial}{\partial x}\Psi(t,0^+)&=\frac{\partial}{\partial x}\Psi(t,0^-)=0,\qquad t>0.
\end{align*}
These boundary conditions are realized in \eqref{Eq_Schroedinger_2} by using the matrix
\begin{equation*}
J=\mmatrix{1}{0}{0}{1},
\end{equation*}
that is, we have $\phi=0$, $\alpha=1$, and $\beta=0$ in \eqref{Eq_General_unitary_matrix}, and hence Case I applies.
The coefficients of the Green's function are given by
\begin{equation*}
\omega_\pm=0,\qquad\mu_\pm^{(x,y)}=0,\quad\text{and}\quad\mu_0^{(x,y)}=\sgn(xy),
\end{equation*}
and lead to
\begin{equation*}
G(t,x,y)=\frac{1}{2\sqrt{i\pi t}}\left(e^{-\frac{(x-y)^2}{4it}}+\sgn(xy)e^{-\frac{(|x|+|y|)^2}{4it}}\right).
\end{equation*}
\end{bsp}

In the next example we consider Robin boundary conditions at the origin. The Neumann boundary conditions in Example~\ref{bsp_Neumann} are contained as
a special case and the Dirichlet boundary conditions in Example~\ref{bsp_Dirichlet} formally appear as a limit; cf. Remark~\ref{remjussi}.

\begin{bsp}[Robin boundary conditions]\label{bsp_Robin}
We consider the free Schr\"odinger equation on the two half lines $\mathbb{R}\setminus\{0\}$ with Robin boundary conditions
\begin{align*}
i\frac{\partial}{\partial t}\Psi(t,x)&=-\frac{\partial^2}{\partial x^2}\Psi(t,x),\qquad t>0,\,x\in\mathbb{R}\setminus\{0\}, \\
\frac{\partial}{\partial x}\Psi(t,0^+)&=a\,\Psi(t,0^+),\hspace{1.3cm} t>0, \\
\frac{\partial}{\partial x}\Psi(t,0^-)&=b\,\Psi(t,0^-),\hspace{1.35cm} t>0,
\end{align*}
for some $a,b\in\mathbb{R}$; note that the minus sign for the derivative at $x=0^-$ on the right hand side of \eqref{Eq_Schroedinger_2} is omitted here.
These boundary conditions are realized in \eqref{Eq_Schroedinger_2} by using the matrix
\begin{equation}\label{jrobin}
J=\mmatrix{\frac{i+a}{i-a}}{0}{0}{\frac{i-b}{i+b}},
\end{equation}
which is of the form \eqref{Eq_General_unitary_matrix} with
\begin{equation*}
 \alpha=\sgn(b-a)\frac{(1-ia)(1-ib)}{\sqrt{1+a^2}\sqrt{1+b^2}},\qquad \beta=0,
\end{equation*}
and $\phi\in [0,\pi)$ chosen such that
\begin{equation*}
e^{i\phi}= \sgn(b-a)\frac{(1-ia)(1+ib)}{\sqrt{1+a^2}\sqrt{1+b^2}},
\end{equation*}
where we use $\sgn(0)=1$. One verifies that Case I applies and a (more technical) computation
finally leads to the Green's function
\begin{equation}\label{greenrobin}
\begin{split}
G(t,x,y)=&\left(-a\,\Theta(x)\Theta(y)\Lambda\left(\frac{|x|+|y|}{2\sqrt{it}}+a\sqrt{it}\right)\right. \\
&\quad\left.+b\,\Theta(-x)\Theta(-y)\Lambda\left(\frac{|x|+|y|}{2\sqrt{it}}-b\sqrt{it}\right)\right)e^{-\frac{(|x|+|y|)^2}{4it}} \\
&+\frac{1}{2\sqrt{i\pi t}}\left(\sgn(xy)e^{-\frac{(|x|+|y|)^2}{4it}}+e^{-\frac{(x-y)^2}{4it}}\right).
\end{split}
\end{equation}
\end{bsp}

\begin{bem}\label{remjussi}
It is clear that for $a=b=0$ the boundary condition and Green's function in Example~\ref{bsp_Robin} reduces to those in Example~\ref{bsp_Neumann}.
Moreover, also the boundary condition and Green's function for the Dirichlet decoupling in Example~\ref{bsp_Dirichlet}
can be recovered from Example~\ref{bsp_Robin}. In fact, for
$a\to\infty$ and $b\to-\infty$ the matrix $J$ in \eqref{jrobin} tends to the one in \eqref{jdir} and using Lemma~\ref{lem_Lambda_Properties}~(iv) one obtains
the asymptotics
\begin{equation*}
\Lambda\left(\frac{|x|+|y|}{2\sqrt{it}}+a\sqrt{it}\right)\sim\frac{1}{a\sqrt{i\pi t}}\qquad\text{and}
\qquad\Lambda\left(\frac{|x|+|y|}{2\sqrt{it}}-b\sqrt{it}\right)\sim\frac{1}{-b\sqrt{i\pi t}}
\end{equation*}
in \eqref{greenrobin}, which then lead to the Green's function \eqref{Eq_Greenfunction_Dirichlet}.
\end{bem}

\section{Solution of the Schr\"odinger equation with a generalized point interaction}\label{sec_Solution_of_the_Schroedinger_equation}

In this section we continue the theme from Section \ref{sec_Greens_function_of_singular_interactions}, where in Theorem~\ref{satz_Green_function}
it was already shown that the Green's function \eqref{Eq_Greenfunction} satisfies the Schr\"odinger equation \eqref{Eq_DGL_Greenfunction} and the jump
condition \eqref{Eq_Jump_condition_Greenfunction} that represents the generalized point interaction at the origin. Now we turn our attention to the
initial value \eqref{Eq_Schroedinger_3}. This missing part will be provided in Theorem \ref{satz_Wave_function} below. However, the main technical
issue here is to make sense of the integral \eqref{Eq_WavefunctionHH}. Since we want to consider, e.g., plane waves $F(x)=e^{ikx}$ as initial conditions,
we have to deal with integrands that are not absolutely integrable. For this purpose the so-called Fresnel integral, discussed in Lemma \ref{lem_Fresnel_integral},
will be useful. The resulting representation of the integral then also ensures, in a mathematical rigorous way, that the properties \eqref{Eq_DGL_Greenfunction}
and \eqref{Eq_Jump_condition_Greenfunction} of the Green's function $G$ carry over to the respective properties \eqref{Eq_Schroedinger_1}
and \eqref{Eq_Schroedinger_2} of the wave function $\Psi$.

\begin{lem}[Fresnel integral]\label{lem_Fresnel_integral}
Let $f:\Omega\to\mathbb{C}$ be holomorphic on an open set $\Omega\subseteq\mathbb{C}$ which contains the sector
\begin{equation}\label{Eq_Salpha}
S_\alpha:=\Set{z\in\mathbb{C} : 0\leq\Arg(z)\leq\alpha}
\end{equation}
for some $\alpha\in(0,\frac{\pi}{2})$, and assume that $f$ satisfies the estimate
\begin{equation}\label{Eq_Fresnel_boundedness}
|f(z)|\leq A\,e^{-\varepsilon\Im(z^2)},\qquad z\in S_\alpha,
\end{equation}
for some $A\geq 0$ and $\varepsilon>0$. Then we get
\begin{equation}\label{Eq_Fresnel_integral}
\lim\limits_{R\to\infty}\int_0^Rf(y)dy=e^{i\alpha}\int_0^\infty f\big(ye^{i\alpha}\big)dy,
\end{equation}
where the integral on the right hand side is absolutely convergent.
\end{lem}

\begin{proof}
For simplicity we will write $k=\tan(\alpha)>0$. For every $R>0$ we consider the integration path

\begin{minipage}{0.5\textwidth}
\begin{align*}
\gamma_1&:=\Set{y : 0\leq y\leq R}, \\
\gamma_2&:=\Set{R+iy : 0\leq y\leq kR}, \\
\gamma_3&:=\Set{ye^{i\alpha} : 0\leq y\leq R\sqrt{1+k^2}}.
\end{align*}
\end{minipage}
\hspace{1cm}
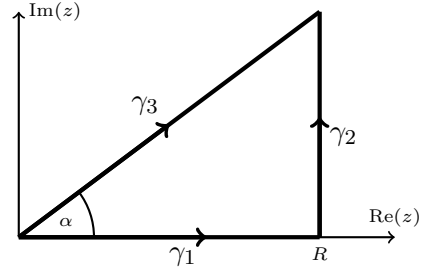
\begin{minipage}{0.39\textwidth}
\begin{tikzpicture}
\draw[thick,->] (0,0)--(5,0);
\draw[thick,->] (0,0)--(0,3);
\draw[ultra thick,->] (0,0)--(2.5,0);
\draw[ultra thick,-] (0,0)--(4,0);
\draw[ultra thick,->] (4,0)--(4,1.6);
\draw[ultra thick,-] (4,0)--(4,3);
\draw[ultra thick,->] (0,0)--(2,1.5);
\draw[ultra thick,-] (0,0)--(4,3);
\draw[thick,-] (1,0)arc(0:37:1);
\filldraw[black] (2.5,0) node[anchor=north east] {$\gamma_1$};
\filldraw[black] (4,1.6) node[anchor=north west] {$\gamma_2$};
\filldraw[black] (2,1.5) node[anchor=south east] {$\gamma_3$};
\filldraw[black] (5,0) node[anchor=south] {\tiny{$\Re(z)$}};
\filldraw[black] (0,3) node[anchor=west] {\tiny{$\Im(z)$}};
\filldraw[black] (0.4,0.2) node[anchor=west] {\tiny{$\alpha$}};
\filldraw[black] (4,0) node[anchor=north] {\tiny{$R$}};
\end{tikzpicture}
\end{minipage}
Since $f$ is holomorphic, Cauchy's theorem yields
\begin{equation}\label{Eq_Fresnel_integral1}
\begin{split}
\int_0^Rf(y)dy=\int_{\gamma_1}f(z)dz&=-\int_{\gamma_2}f(z)dz+\int_{\gamma_3}f(z)dz\\
&=-i\int_0^{kR}f(R+iy)dy + e^{i\alpha}\int_0^{R\sqrt{1+k^2}}f\big(ye^{i\alpha}\big)dy.
\end{split}
\end{equation}
From the estimate \eqref{Eq_Fresnel_boundedness} we obtain
\begin{equation*}
\left|-i\int_0^{kR}f(R+iy)dy\right|\leq A\int_0^\infty e^{-2\varepsilon Ry}dy=\frac{A}{2\varepsilon R}\rightarrow 0,\quad R\rightarrow\infty,
\end{equation*}
and thus in the limit $R\rightarrow\infty$ we conclude from \eqref{Eq_Fresnel_integral1}
\begin{equation*}
\lim\limits_{R\to\infty}\int_0^Rf(y)dy=e^{i\alpha}\lim\limits_{R\to\infty}\int_0^Rf\big(ye^{i\alpha}\big)dy.
\end{equation*}
The estimate $|f(ye^{i\alpha})|\leq Ae^{-\varepsilon\sin(2\alpha)y^2}$, $y>0$, implies that
the integral on the right hand side is absolutely convergent and hence the identity \eqref{Eq_Fresnel_integral} follows.
\end{proof}

In the next lemma we define functions $\Psi_0$, $\Psi_1$, and $\Psi_\text{free}$ that are closely related to
the functions $G_0$, $G_1$, and $G_\text{free}$ in \eqref{Eq_Part_Greenfunction}, which will then lead to a solution of the Schr\"odinger equation \eqref{Eq_Schroedinger}
in Theorem~\ref{satz_Wave_function} below.

\begin{lem}\label{lem_Psi}
Let $F:\Omega\to\mathbb{C}$ be holomorphic on an open set $\Omega\subseteq\mathbb{C}$ which contains the sector $S_\alpha$ from \eqref{Eq_Salpha}
for some $\alpha\in(0,\frac{\pi}{2})$, and assume that $F$ satisfies the estimate
\begin{equation}\label{Eq_Exponential_boundedness}
|F(z)|\leq A\,e^{B\Im(z)},\qquad z\in S_\alpha,
\end{equation}
for some $A,B\geq 0$. For every fixed $t>0$, $x\in\mathbb{R}\setminus\{0\}$ we consider the function
\begin{equation}\label{Eq_Psi}
\Psi_j(t,x;F)=\int_0^\infty G_j(t,x,y)F(y)dy,\qquad j\in\{0,1,\textnormal{free}\}.
\end{equation}
Then the following assertions hold:

\begin{itemize}
\item[{\rm (i)}] The integral on the right hand side in \eqref{Eq_Psi} exists as the improper Riemann integral
\begin{equation}\label{Eq_Psi2}
\int_0^\infty G_j(t,x,y)F(y)dy:=\lim_{R\rightarrow\infty}\int_0^RG_j(t,x,y)F(y)dy,
\end{equation}
and the function $\Psi_j$ admits the absolute integrable representation
\begin{equation}\label{Eq_Psi3}
\Psi_j(t,x;F)=e^{i\alpha}\int_0^\infty G_j(t,x,ye^{i\alpha})F(ye^{i\alpha})dy,\qquad j\in\{0,1,\textnormal{free}\}.
\end{equation}

\item [{\rm (ii)}] The functions $\Psi_j$, $j\in\{0,1,\textnormal{free}\}$, in \eqref{Eq_Psi} are solutions of the differential equation
\begin{equation}\label{Eq_Psi_DGL}
i\frac{\partial}{\partial t}\Psi_j(t,x;F)=-\frac{\partial^2}{\partial x^2}\Psi_j(t,x;F),\qquad t>0,\,x\in\mathbb{R}\setminus\{0\}.
\end{equation}
\item [{\rm (iii)}] The functions $\Psi_j$, $ j\in\{0,1,\textnormal{free}\}$, in \eqref{Eq_Psi} admit the initial values
\begin{equation}\label{Eq_Psi_initial_condition1}
\Psi_0(0^+,x;F)=\Psi_1(0^+,x;F)=0,\qquad x\in\mathbb{R}\setminus\{0\},
\end{equation}
and
\begin{equation}\label{Eq_Psi_initial_condition2}
\Psi_\textnormal{free}(0^+,x;F)=\left\{\begin{array}{ll} F(x), & \textnormal{if }x>0, \\ 0, & \textnormal{if }x<0. \end{array}\right.
\end{equation}
\end{itemize}
\end{lem}

\begin{proof}
(i) This assertion is a direct consequence of Lemma~\ref{lem_Fresnel_integral} if we verify that the functions $y\mapsto G_j(t,x,y)F(y)$, $ j\in\{0,1,\textnormal{free}\}$, satisfy an estimate of the form \eqref{Eq_Fresnel_boundedness}. In fact, the estimate \eqref{Eq_Estimate_1} together with the assumption \eqref{Eq_Exponential_boundedness} leads to the bound
\begin{equation}\label{Eq_Estimate_integrand}
\big|G_j(t,x,z)F(z)\big|\leq Ac_j(t)e^{-\frac{\Im(z^2)}{4t}+(B+\frac{|x|}{2t})\Im(z)},\qquad z\in S_\alpha.
\end{equation}
Since for every $z\in S_\alpha$ we have $\Im(z)\leq\tan(\alpha)\Re(z)$, and hence $\Im(z^2)\geq\frac{2}{\tan(\alpha)}\Im(z)^2$,
the exponent in \eqref{Eq_Estimate_integrand} can be further estimated by
\begin{equation*}
-\frac{\Im(z^2)}{4t}+\Big(B+\frac{|x|}{2t}\Big)\Im(z)\leq-\frac{\Im(z^2)}{8t}-\frac{\Im(z)^2}{4t\tan(\alpha)}+\Big(B+\frac{|x|}{2t}\Big)\Im(z).
\end{equation*}
Taking into account that a polynomial of the form $-a\Im(z)^2+b\Im(z)$ with $a>0$, $b\in\mathbb{R}$, is bounded by $\frac{b^2}{4a}$, we
find
\begin{equation}\label{Eq_Estimate_Exponent2}
-\frac{\Im(z^2)}{4t}+\Big(B+\frac{|x|}{2t}\Big)\Im(z)\leq-\frac{\Im(z^2)}{8t}+t\Big(B+\frac{|x|}{2t}\Big)^2\tan(\alpha),
\end{equation}
and thus \eqref{Eq_Estimate_integrand} can be estimated by
\begin{equation*}
\big|G_j(t,x,z)F(z)\big|\leq Ac_j(t)e^{t(B+\frac{|x|}{2t})^2\tan(\alpha)}e^{-\frac{\Im(z^2)}{8t}},\qquad z\in S_\alpha.
\end{equation*}
This shows that \eqref{Eq_Fresnel_boundedness} indeed holds in the present context and the integral \eqref{Eq_Psi} exists in the form \eqref{Eq_Psi2} and admits the absolute integrable representation \eqref{Eq_Psi3}.

\medskip

\noindent (ii) Now we show that the functions $\Psi_0$, $\Psi_1$, and $\Psi_\text{free}$ satisfy the differential equation \eqref{Eq_Psi_DGL}. Since we have already
shown in Lemma \ref{lem_Part_Greenfunction_DGL} that $G_0$, $G_1$, and $G_\text{free}$ solve \eqref{Eq_Part_Greenfunction_DGL},
it remains to interchange the integral and the derivatives in the representation \eqref{Eq_Psi3}. We verify this property for the
time derivative of $G_0$ and leave the analog arguments for the spatial derivatives and the functions
$G_1$ and $G_\text{free}$ to the reader.
Note that from \eqref{Eq_Derivative_G0_1} with $z=ye^{i\alpha}$ one obtains
\begin{equation*}
\Big|\frac{\partial}{\partial t}G_0\big(t,x,ye^{i\alpha}\big)\Big|\leq\frac{1}{4t\sqrt{\pi t}}\Big(1+\frac{(|x|+|y|)^2}{2t}\Big)e^{-\frac{y^2\sin(2\alpha)}{4t}-\frac{|x|y\sin(\alpha)}{2t}}
\end{equation*}
and hence together with \eqref{Eq_Exponential_boundedness}
\begin{equation*}
\Big|\frac{\partial}{\partial t}G_0(t,x,ye^{i\alpha})F(ye^{i\alpha})\Big|\leq\frac{A}{4t\sqrt{\pi t}}\Big(1+\frac{(|x|+|y|)^2}{2t}\Big)
e^{-\frac{y^2\sin(2\alpha)}{4t}-\frac{|x|y\sin(\alpha)}{2t}+By\sin (\alpha)}.
\end{equation*}
The term $e^{-\frac{y^2\sin(2\alpha)}{4t}}$ now ensures the integrability of the right hand side. Since all terms are continuous functions in $t$, we can also choose
an integrable upper bound, which is locally uniform in $t$. Hence, by classical theorems for Lebesgue integral (see, e.g.,
\cite{Royden})
 the time derivative of $\Psi_0(t,x;F)$ exists and is given by
\begin{equation*}
\frac{\partial}{\partial t}\Psi_0(t,x;F)=e^{i\alpha}\int_0^\infty\frac{\partial}{\partial t}G_0(t,x,ye^{i\alpha})F(ye^{i\alpha})dy.
\end{equation*}
Similar arguments also apply to the other eight derivatives in \eqref{Eq_Derivative_G0}, \eqref{Eq_Derivative_G1}, and \eqref{Eq_Derivative_Gfree},
and we conclude for every $j\in\{0,1,\text{free}\}$
\begin{subequations}\label{Eq_Derivative_Psi}
\begin{align}
\frac{\partial}{\partial t}\Psi_j(t,x;F)&=e^{i\alpha}\int_0^\infty\frac{\partial}{\partial t}G_j\big(t,x,ye^{i\alpha}\big)F\big(ye^{i\alpha}\big)dy, \label{Eq_Derivative_Psi_t} \\
\frac{\partial}{\partial x}\Psi_j(t,x;F)&=e^{i\alpha}\int_0^\infty\frac{\partial}{\partial x}G_j\big(t,x,ye^{i\alpha}\big)F\big(ye^{i\alpha}\big)dy, \label{Eq_Derivative_Psi_x} \\
\frac{\partial^2}{\partial x^2}\Psi_j(t,x;F)&=e^{i\alpha}\int_0^\infty\frac{\partial^2}{\partial x^2}G_j\big(t,x,ye^{i\alpha}\big)F\big(ye^{i\alpha}\big)dy.\label{Eq_Derivative_Psi_xx}
\end{align}
\end{subequations}
As was already mentioned the functions $G_j$ solve \eqref{Eq_Part_Greenfunction_DGL} and hence it follows that
the functions $\Psi_j$ satisfy \eqref{Eq_Psi_DGL}.

\medskip

\noindent (iii) To check the initial conditions \eqref{Eq_Psi_initial_condition1} for $\Psi_0$ and $\Psi_1$, we plug in the estimates \eqref{Eq_Estimate_2} and \eqref{Eq_Exponential_boundedness} into the representation \eqref{Eq_Psi3}. This yields
\begin{equation}\label{Eq_Initial_value_Psi0}
\begin{split}
\big|\Psi_j(t,x;F)\big|&\leq Ac_j(t)\int_0^\infty e^{-\frac{y^2\sin(2\alpha)}{4t}+(B-\frac{|x|}{2t})y\sin(\alpha)}dy \\
&=\frac{Ac_j(t)\sqrt{\pi t}}{\sqrt{\sin(2\alpha)}}\Lambda\left(\left(\frac{|x|}{2\sqrt{t}}-B\sqrt{t}\right)\sqrt{\frac{\tan(\alpha)}{2}}\right)
\rightarrow 0,\qquad t\rightarrow 0^+,
\end{split}
\end{equation}
for $j\in\{0,1\}$,
where we have used the integral
 \eqref{Eq_Lambda_Integral} in the second line; the convergence follows from the asymptotics \eqref{Eq_Lambda_Asymptotic}
and the fact that $c_j(t)\sqrt{t}$ is bounded (for the precise form of the constants see Lemma \ref{cor_G_Estimate}).
For the initial value of $\Psi_\text{free}$ we distinguish two cases. For $x<0$ we use the estimate \eqref{Eq_Estimate_3} and
get the same convergence as in \eqref{Eq_Initial_value_Psi0}. The remaining case $x>0$ is more involved. Here we split up the integral \eqref{Eq_Psi} into
\begin{equation}\label{Eq_Psifree_split}
\Psi_\text{free}(t,x;F)=\frac{1}{2\sqrt{i\pi t}}\left(\int_0^{2x}e^{-\frac{(x-y)^2}{4it}}F(y)dy+\int_{2x}^\infty e^{-\frac{(x-y)^2}{4it}}F(y)dy\right).
\end{equation}
In the first integral we use the derivative $\frac{d}{dz}\erf(z)=\frac{2}{\sqrt{\pi}}e^{-z^2}$ of the error function, as well as integration by parts, to get
\begin{align*}
&\frac{1}{2\sqrt{i\pi t}}\int_0^{2x}e^{-\frac{(x-y)^2}{4it}}F(y)dy=-\frac{1}{2}\int_0^{2x}\frac{d}{dy}\erf\left(\frac{x-y}{2\sqrt{it}}\right)F(y)dy \\
&\hspace{1cm}=\frac{1}{2}\left(\erf\left(\frac{x}{2\sqrt{it}}\right)F(0)-\erf\left(\frac{-x}{2\sqrt{it}}\right)F(2x)+\int_0^{2x}
\erf\left(\frac{x-y}{2\sqrt{it}}\right)F'(y)dy\right).
\end{align*}
Using $\lim_{t\to 0^+}\erf\big(\frac{\xi}{2\sqrt{it}}\big)=\sgn(\xi)$, $\xi\in\mathbb{R}$, and dominated convergence we get
\begin{equation*}
\lim\limits_{t\to 0^+}\frac{1}{2\sqrt{i\pi t}}\int_0^{2x}e^{-\frac{(x-y)^2}{4it}}F(y)dy=\frac{1}{2}\left(F(0)+F(2x)+\int_0^{2x}\sgn(x-y)F'(y)dy\right)=F(x).
\end{equation*}
In the second integral in \eqref{Eq_Psifree_split} we substitute $y\to y+2x$ and obtain
\begin{align*}
\frac{1}{2\sqrt{i\pi t}}\int_{2x}^\infty e^{-\frac{(x-y)^2}{4it}}F(y)dy=\frac{1}{2\sqrt{i\pi t}}\int_0^\infty e^{-\frac{(x+y)^2}{4it}}F(y+2x)dy.
\end{align*}
This is the same integral as the one for $\Psi_0$, with the initial function $F(\,\cdot\,+2x)$ instead $F$. Consequently,
this integral also vanishes in the limit $t\to 0^+$. Thus, we have also shown the initial condition \eqref{Eq_Psi_initial_condition2}
for $\Psi_\text{free}$.
\end{proof}

As the last preparatory statement we prove the following lemma about the representation of the
functions $\Psi_0$, $\Psi_1$, and $\Psi_\text{free}$ at the support of the singular interaction $x=0^\pm$.

\begin{lem}\label{lem_Psi_x0}
Let $F:\Omega\to\mathbb{C}$ be holomorphic on an open set $\Omega\subseteq\mathbb{C}$ which contains the sector $S_\alpha$ from \eqref{Eq_Salpha}
for some $\alpha\in(0,\frac{\pi}{2})$, and assume that $F$ satisfies the estimate
\begin{equation}\label{Eq_Exponential_boundedness2}
|F(z)|\leq A\,e^{B\Im(z)},\qquad z\in S_\alpha,
\end{equation}
for some $A,B\geq 0$. Then for the functions $\Psi_j$, $j\in\{0,1,\textnormal{free}\}$, from \eqref{Eq_Psi} and their spatial derivatives we are allowed to carry the limit $x\to 0^\pm$ inside the integral
\begin{subequations}\label{Eq_Psi_x0}
\begin{align}
\Psi_j(t,0^\pm;F)&=\int_0^\infty G_j(t,0^\pm,y)F(y)dy, \label{Eq_Psi_x0_1} \\
\frac{\partial}{\partial x}\Psi_j(t,0^\pm;F)&=\int_0^\infty\frac{\partial}{\partial x}G_j(t,0^\pm,y)F(y)dy, \label{Eq_Psi_x0_2}
\end{align}
\end{subequations}
where, similar to \eqref{Eq_Psi2}, the integrals exist as improper Riemann integrals
$\int_0^\infty \coloneqq\lim\limits_{R\to\infty}\int_0^R$.
\end{lem}

\begin{proof}
For the function $\Psi_j$, in the representation \eqref{Eq_Psi3}, we have the estimate
\begin{equation}\label{Eq_Estimate_x0}
\big|G_j(t,x,ye^{i\alpha})F(ye^{i\alpha})\big|\leq Ac_j(t)e^{-\frac{y^2\sin(2\alpha)}{4t}+y(B+\frac{|x|}{2t})\sin(\alpha)},
\end{equation}
which follows from the assumption \eqref{Eq_Exponential_boundedness2} on $F$ and \eqref{Eq_Estimate_1}. Since this upper bound is continuous in $x$, we can choose it to be uniform for all $x$ in a neighborhood of $0$. Now we can use dominated convergence in \eqref{Eq_Psi3} to get the absolute integrable representation
\begin{equation}\label{Eq_Psij_limit_absolute_integrable}
\Psi_j(t,0^\pm;F)=e^{i\alpha}\int_0^\infty G_j(t,0^\pm,ye^{i\alpha})F(ye^{i\alpha})dy.
\end{equation}
Once more from \eqref{Eq_Exponential_boundedness2} and \eqref{Eq_Estimate_1} we get the estimate
\begin{equation*}
\big|G_j(t,0^\pm,z)F(z)\big|\leq Ac_j(t)\,e^{-\frac{\Im(z^2)}{4t}+B\Im(z)},\qquad z\in S_\alpha.
\end{equation*}
The estimate \eqref{Eq_Estimate_Exponent2} for $x=0$ allows to further estimate the integrand by
\begin{equation}\label{Eq_Estimate_x02}
\big|G_j(t,0^\pm,z)F(z)\big|\leq Ac_j(t)\,e^{-\frac{\Im(z^2)}{8t}+tB^2\tan(\alpha)},\qquad z\in S_\alpha.
\end{equation}
This estimate shows, in particular, that the assumption \eqref{Eq_Fresnel_boundedness} of Lemma \ref{lem_Fresnel_integral}
is satisfied and hence we can use \eqref{Eq_Fresnel_integral} to rewrite
the absolute integrable representation \eqref{Eq_Psij_limit_absolute_integrable} into the improper Riemann integral \eqref{Eq_Psi_x0_1}.

\medskip

The same argument applies also to the spatial derivative in \eqref{Eq_Derivative_Psi_x}. Here, the explicit representations \eqref{Eq_Derivative_G0_2}, \eqref{Eq_Derivative_G1_2}, and \eqref{Eq_Derivative_Gfree_2} lead to a similar estimate as in \eqref{Eq_Estimate_1}, and consequently also to estimates of the form \eqref{Eq_Estimate_x0} and \eqref{Eq_Estimate_x02}.
\end{proof}

The next theorem is the main result of this section, where a solution $\Psi$ of the Schr\"odinger equation \eqref{Eq_Schroedinger}
is obtained by assembling the components $\Psi_j$ from \eqref{Eq_Psi} based on the structure of the Green's function
in \eqref{Eq_Greenfunction_decomposition}. Besides the four parts of the Green's function we also have to consider
that now integrals over $\mathbb{R}$ appear, whereas the integrals in \eqref{Eq_Psi} are only over the positive half line $(0,\infty)$.

\begin{satz}\label{satz_Wave_function}
Let $F:\Omega\to\mathbb{C}$ be holomorphic on an open set $\Omega\subseteq\mathbb{C}$ which contains the double sector
\begin{equation}\label{Eq_Doublesector}
S_\alpha\cup(-S_\alpha)=\Set{z\in\mathbb{C} : \Arg(z)\in[0,\alpha]\cup[\pi,\pi+\alpha]}
\end{equation}
for some $\alpha\in(0,\frac{\pi}{2})$, and assume that $F$ satisfies the estimate
\begin{equation}\label{Eq_Exponential_boundedness_doublesector}
|F(z)|\leq A\,e^{B|\Im(z)|},\qquad z\in S_\alpha\cup(-S_\alpha),
\end{equation}
for some $A,B\geq 0$. Let $G$ be the Green's function in \eqref{Eq_Greenfunction} or \eqref{Eq_Greenfunction_decomposition}. Then the function
\begin{equation}\label{Eq_Psiqq}
\Psi(t,x;F)=\int_\mathbb{R}G(t,x,y)F(y)dy,\qquad t>0,\,x\in\mathbb{R}\setminus\{0\},
\end{equation}
exists as an improper Riemann integral of the form
\begin{equation}\label{Eq_Psiq}
\int_\mathbb{R}G(t,x,y)F(y)dy:=\lim\limits_{R_1\to\infty}\int_{-R_1}^0G(t,x,y)F(y)dy+\lim\limits_{R_2\to\infty}\int_0^{R_2}G(t,x,y)F(y)dy
\end{equation}
and $\Psi$ is a solution of the Schr\"odinger equation \eqref{Eq_Schroedinger}.
\end{satz}

\begin{proof}
For $y>0$ the Green's function \eqref{Eq_Greenfunction_decomposition} can be written as
\begin{equation*}
G(t,x,y)=\mu_+^{(x,0^+)}G_1(t,x,y;\omega_+)+\mu_-^{(x,0^+)}G_1(t,x,y;\omega_-)+\mu_0^{(x,0^+)}G_0(t,x,y)+G_\text{free}(t,x,y).
\end{equation*}
Hence we conclude from Lemma~\ref{lem_Psi}~(i) that the limit
\begin{equation*}
\begin{split}
\lim\limits_{R_2\to\infty}\int_0^{R_2}G(t,x,y)F(y)dy=&\mu_+^{(x,0^+)}\Psi_1(t,x;\omega_+,F)+\mu_-^{(x,0^+)}\Psi_1(t,x;\omega_-,F) \\
&+\mu_0^{(x,0^+)}\Psi_0(t,x;F)+\Psi_\text{free}(t,x;F)
\end{split}
\end{equation*}
exists. Moreover, for $y>0$ we also have
\begin{equation*}
G(t,x,-y)=\mu_+^{(x,0^-)}G_1(t,x,y;\omega_+)+\mu_-^{(x,0^-)}G_1(t,x,y;\omega_-)+\mu_0^{(x,0^-)}G_0(t,x,y)+G_\text{free}(t,-x,y),
\end{equation*}
where we used $G_\text{free}(t,x,-y)=G_\text{free}(t,-x,y)$, a direct consequence of \eqref{Eq_Part_Greenfunctionfree}. Again from Lemma \ref{lem_Psi}~(i) we conclude that also the limit
\begin{equation*}
\begin{split}
\lim\limits_{R_1\rightarrow\infty}\int_{-R_1}^0G(t,x,y)F(y)dy&=\lim\limits_{R_1\rightarrow\infty}\int_0^{R_1}G(t,x,-y)\widetilde{F}(y)dy \\
&=\mu_+^{(x,0^-)}\Psi_1(t,x;\omega_+,\widetilde{F})+\mu_-^{(x,0^-)}\Psi_1(t,x;\omega_-,\widetilde{F}) \\
&\quad+\mu_0^{(x,0^-)}\Psi_0(t,x;\widetilde{F})+\Psi_\text{free}(t,-x;\widetilde{F})
\end{split}
\end{equation*}
exists. Here we used the mirrored function $\widetilde{F}(z):= F(-z)$, which also satisfies the assumption \eqref{Eq_Exponential_boundedness},
since \eqref{Eq_Exponential_boundedness_doublesector} holds on the double sector $S_\alpha\cup(-S_\alpha)$.
This leads to the existence of the function $\Psi$ in \eqref{Eq_Psiqq} in the sense of \eqref{Eq_Psiq}, and also shows that it can be decomposed into
\begin{equation}\label{Eq_Psi_decomposition}
\begin{split}
\Psi(t,x;F)&=\mu_+^{(x,0^-)}\Psi_1(t,x;\omega_+,\widetilde{F})+\mu_+^{(x,0^+)}\Psi_1(t,x;\omega_+,F) \\
&\quad +\mu_-^{(x,0^-)}\Psi_1(t,x;\omega_-,\widetilde{F})+\mu_-^{(x,0^+)}\Psi_1(t,x;\omega_-,F) \\
&\quad +\mu_0^{(x,0^-)}\Psi_0(t,x;\widetilde{F})+\mu_0^{(x,0^+)}\Psi_0(t,x;F) \\
&\quad +\Psi_\text{free}(t,-x;\widetilde{F})+\Psi_\text{free}(t,x;F).
\end{split}
\end{equation}
Due to \eqref{Eq_Psi_DGL} the functions $\Psi_0$, $\Psi_1$, and $\Psi_\text{free}$ are solutions of the differential equation, and so is its linear combination $\Psi$ a solution of \eqref{Eq_Schroedinger_1}. Note, that the coefficients $\mu_\pm$ and $\mu_0$ only depend on the sign of $x$ and hence do not influence the differential equation. Moreover, although the term $\Psi_\text{free}(t,-x,\widetilde{F})$ depends on the variable $-x$, this function also solves \eqref{Eq_Schroedinger_1} since the $x$-derivative is of second order.

\medskip

In order to check the jump condition \eqref{Eq_Schroedinger_2} we notice that by Lemma \ref{lem_Psi_x0} we are allowed
to carry the limit $x\to 0^\pm$ inside the integral. Hence we get the representations
\begin{align*}
\Psi(t,0^\pm;F)&=\int_\mathbb{R}G(t,0^\pm,y)F(y)dy, \\
\frac{\partial}{\partial x}\Psi(t,0^\pm;F)&=\int_\mathbb{R}\frac{\partial}{\partial x}G(t,0^\pm,y)F(y)dy,
\end{align*}
also for the linear combination. Again, note that the negative $x$ argument of $\Psi_\text{free}(t,-x;\widetilde{F})$ does not matter, since $G_\text{free}(t,0^+,y)=G_\text{free}(t,0^-,y)$
by definition \eqref{Eq_Part_Greenfunctionfree}. Since $G$ satisfies the jump condition \eqref{Eq_Jump_condition_Greenfunction},
the function $\Psi$ satisfies the jump condition \eqref{Eq_Schroedinger_2}.
Finally, the initial values \eqref{Eq_Psi_initial_condition1} and \eqref{Eq_Psi_initial_condition2} imply the initial condition \eqref{Eq_Schroedinger_3}
of the wave function $\Psi$.
\end{proof}

In preparation for the analysis of superoscillations in the next section
we will now briefly discuss convergent sequences of initial conditions $(F_n)_n$ and the
convergence of the corresponding solutions $(\Psi(t,x;F_n))_n$ of the Schr\"odinger equation \eqref{Eq_Schroedinger}.
As before we shall first deal with the functions $\Psi_j$, $j\in\{0,1,\text{free}\}$, in \eqref{Eq_Psi} and assemble these
components afterwards to the whole wave function $\Psi$; cf. \eqref{Eq_Psi_decomposition} in the proof of Theorem~\ref{satz_Wave_function}.

\begin{lem}\label{lem_Psij_convergence}
Let $F,F_n:\Omega\rightarrow\mathbb{C}$, $n\in\mathbb{N}_0$, be holomorphic on an open set $\Omega\subseteq\mathbb{C}$
which contains the sector $S_\alpha$ from \eqref{Eq_Salpha} for some $\alpha\in(0,\frac{\pi}{2})$, and assume that
for some $A,B\geq 0$ and $A_n,B_n\geq 0$, $n\in\mathbb{N}_0$, the exponential bounds \eqref{Eq_Exponential_boundedness} hold. If the sequence $(F_n)_n$ converges as
\begin{equation}\label{Eq_Salpha_convergence}
\lim\limits_{n\rightarrow\infty}\sup\limits_{z\in S_\alpha}\big|F_n(z)-F(z)\big|e^{-C|z|}=0
\end{equation}
for some $C\geq 0$, then also the corresponding wave functions $\Psi_j$, $j\in\{0,1,\textnormal{free}\}$, in \eqref{Eq_Psi} converge as
\begin{equation}\label{Eq_Psij_convergence}
\lim\limits_{n\rightarrow\infty}\Psi_j(t,x;F_n)=\Psi_j(t,x;F),\qquad j\in\{0,1,\textnormal{free}\},
\end{equation}
uniformly on compact subsets of $(0,\infty)\times\mathbb{R}$.
\end{lem}

\begin{proof}
First of all, we have the estimate
\begin{equation*}
|F_n(z)-F(z)|\leq C_ne^{C|z|},\qquad z\in S_\alpha,
\end{equation*}
where $C_n:=\sup_{z\in S_\alpha}|F_n(z)-F(z)|e^{-C|z|}$. Using the representation \eqref{Eq_Psi3},
this inequality together with the estimate \eqref{Eq_Estimate_1} of the
Green's function, leads to
\begin{align*}
\big|\Psi_j(t,x;F_n)-\Psi_j(t,x;F)\big|&=\Big|\int_0^\infty G_j(t,x,ye^{i\alpha})\big(F_n(ye^{i\alpha})-F(ye^{i\alpha})\big)dy\Big| \\
&\leq C_nc_j(t)\int_0^\infty e^{-\frac{y^2\sin(2\alpha)}{4t}+\frac{|x|y\sin(\alpha)}{2t}}e^{Cy}dy \\
&=C_n\frac{c_j(t)\sqrt{\pi t}}{\sqrt{\sin(2\alpha)}}\Lambda\Big(-\frac{|x|\sqrt{\tan(\alpha)}}{2\sqrt{2t}}-\frac{C\sqrt{t}}{\sqrt{\sin(2\alpha)}}\Big),
\end{align*}
where in the last line we used the integral \eqref{Eq_Lambda_Integral}. Since the right hand side of this inequality is
continuous in $t\in (0,\infty)$ and $x\in\mathbb R$, and we have $C_n\to 0$ by the assumption \eqref{Eq_Salpha_convergence},
the uniform convergence \eqref{Eq_Psij_convergence} on compact subsets of $(0,\infty)\times\mathbb{R}$ follows.
\end{proof}

Lemma~\ref{lem_Psij_convergence} now leads to the following theorem, which is an important ingredient in the next section.

\begin{satz}\label{satz_Psi_convergence}
Let $F,F_n:\Omega\rightarrow\mathbb{C}$, $n\in\mathbb N_0$, be holomorphic on an open set $\Omega\subseteq\mathbb{C}$
which contains the double sector $S_\alpha\cup(-S_\alpha)$ from \eqref{Eq_Doublesector} for some $\alpha\in(0,\frac{\pi}{2})$, and assume that
for some $A,B\geq 0$ and $A_n,B_n\geq 0$, $n\in\mathbb N_0$, the exponential bounds \eqref{Eq_Exponential_boundedness_doublesector} hold.
If the sequence $(F_n)_n$ converges as
\begin{equation}\label{Eq_Doublesector_convergence}
\lim\limits_{n\rightarrow\infty}\sup\limits_{z\in S_\alpha\cup(-S_\alpha)}\big|F_n(z)-F(z)\big|e^{-C|z|}=0
\end{equation}
for some $C\geq 0$, then also the corresponding wave functions $\Psi$ in \eqref{Eq_Psiqq} converge as
\begin{equation}\label{Eq_Psi_convergence}
\lim\limits_{n\rightarrow\infty}\Psi(t,x;F_n)=\Psi(t,x;F),
\end{equation}
uniformly on compact subsets of $(0,\infty)\times\mathbb{R}$.
\end{satz}

\begin{proof}
It is clear that the convergence \eqref{Eq_Doublesector_convergence} of the functions $F_n$, $n\in\mathbb N_0$,
implies the same convergence of the mirrored functions $\widetilde{F}_n(z)=F_n(-z)$, $n\in\mathbb N_0$, and $\widetilde{F}(z)=F(-z)$.
Since the function $\Psi$ can be decomposed in the form \eqref{Eq_Psi_decomposition},
the convergence \eqref{Eq_Psi_convergence} follows immediately from Lemma \ref{lem_Psij_convergence}.
\end{proof}

\section{Superoscillatory initial data and plane wave asymptotics}\label{sec_Time_evolution_of_superoscillations}

In this section we allow superoscillatory functions as initial data in the Schr\"odinger equation \eqref{Eq_Schroedinger} and we
show that the corresponding solutions converge uniformly on compact sets. To discuss the oscillatory properties of these solutions
we study the long time asymptotics of the plane
wave solution in Theorem~\ref{satz_Asymptotics} and Remark~\ref{lastrem}, where the expected oscillatory behaviour and also possible stationary terms
reflecting negative bound states of the singular potential are identified.

Superoscillating functions are band-limited functions that can oscillate faster than their fastest Fourier component. This is made precise in the next definition.

\begin{defi}[Superoscillations]\label{defi_Superoscillations}
A {\em generalized Fourier sequence} is a sequence of functions $(F_n)_n$, $n\in\mathbb{N}_0$, of the form
\begin{equation}\label{basic_sequence}
F_n(x)= \sum_{j=0}^n c_j(n)e^{ik_j(n)x},\qquad x\in\mathbb{R},
\end{equation}
with $k_j(n)\in\mathbb{R}$ and $c_j(n)\in\mathbb{C}$, $j\in\{0,\dots,n\}$.
A generalized Fourier sequence $(F_n)_n$ is said to be {\em superoscillating}, if:

\begin{enumerate}
\item[{\rm (i)}] There exists some $k\in\mathbb{R}$ such that $$\sup_{n\in\mathbb{N}_0,\, j\in\{0,\dots,n\}}|k_j(n)|<\vert k\vert.$$

\item[{\rm (ii)}] There exists a compact subset $K\subset \mathbb{R}$, called {\rm superoscillation set}, such that
\begin{equation}\label{Eq_Compact_convergence}
\lim\limits_{n\rightarrow\infty}\sup\limits_{x\in K}\big| F_n(x) - e^{ikx}\big|=0.
\end{equation}
\end{enumerate}
\end{defi}

In the next corollary, which is a simple consequence of Theorem~\ref{satz_Psi_convergence}, it will be shown that superoscillating initial data $(F_n)_n$
(with a slightly stronger convergence property) leads to solutions $(\Psi(t,x;F_n))_n$ that converge on compact subsets for all times $t>0$.
We mention that the characteristic superoscillatory behaviour of the functions $(F_n)_n$ is on a compact set $K$ in \eqref{Eq_Compact_convergence},
but this is not enough to ensure the same convergence
for the sequence of solutions $(\Psi(t,x;F_n))_n$. As the functions \eqref{basic_sequence}
admit entire extensions to the whole complex plane,
\begin{equation}\label{Eq_Entire_extension}
F_n(z)=\sum_{j=0}^n c_j(n)e^{ik_j(n)z},\qquad z\in\mathbb{C},
\end{equation}
a fact that is also important for several considerations related with the so-called supershift property, it is meaningful to assume the
stronger convergence
\begin{equation}\label{Eq_Superoscillations_convergence}
\lim\limits_{n\rightarrow\infty}\sup\limits_{z\in\mathbb{C}}\big|F_n(z)-e^{ikz}\big|e^{-C|z|}=0
\end{equation}
for some $C\geq 0$.
Note, that in our standard example of superoscillating functions in \eqref{Eq_Fn} one indeed has this kind of
uniform convergence; cf. \cite[Theorem 2.1]{CSSYgenfun} and \cite{acsst5} for more details.

\begin{cor}[Stability of superoscillations]\label{cor_Stability_of_superoscillations}
Let the sequence $(F_n)_n$, $n\in\mathbb{N}_0$,  be superoscillating in the sense of Definition \ref{defi_Superoscillations} and assume, in addition,
that their entire extensions \eqref{Eq_Entire_extension} converge as in \eqref{Eq_Superoscillations_convergence}
for some $C\geq 0$. Then also the corresponding solutions of \eqref{Eq_Schroedinger} converge as
\begin{equation*}
\lim\limits_{n\rightarrow\infty}\Psi(t,x;F_n)=\Psi(t,x;e^{ik\,\cdot\,}),
\end{equation*}
uniformly on compact subsets of $(0,\infty)\times\mathbb{R}$.
\end{cor}

\begin{proof}
In order to apply Theorem \ref{satz_Psi_convergence}, we first note that \eqref{Eq_Entire_extension} implies the estimate
\begin{equation*}
|F_n(z)|\leq\sum\limits_{j=0}^n|c_j(n)|e^{\Re(ik_j(n)z)}\leq\sum\limits_{j=0}^n|c_j(n)|e^{|k_j(n)|\,|\Im(z)|},\qquad z\in\mathbb{C}.
\end{equation*}
Together with the convergence \eqref{Eq_Superoscillations_convergence} this means that the functions $F_n$
satisfy the assumptions of Theorem~\ref{satz_Psi_convergence} for any $\alpha\in(0,\frac{\pi}{2})$, and hence the statement follows.
\end{proof}

To analyse the oscillatory behaviour of the functions $\Psi(t,x;F_n)$ and $\Psi(t,x;e^{ik\,\cdot\,})$ in Corollary~\ref{cor_Stability_of_superoscillations}
it is useful compute the explicit form of the
plane wave solution $\Psi(t,x;e^{ik\,\cdot\,})$ and to provide its long time asymptotics.

\begin{prop}\label{prop_Plane_wave_solution}
For every $k\in\mathbb{R}$ the solution of the Schr\"odinger equation \eqref{Eq_Schroedinger} with initial condition $F(x)=e^{ikx}$ is given by
\begin{equation}\label{Eq_Plane_wave_solution}
\begin{split}
\Psi(t,x;e^{ik\,\cdot})=&\left(\frac{\mu_+^{(x,0^+)}}{\omega_++ik}+\frac{\mu_-^{(x,0^+)}}{\omega_-+ik}+\frac{\mu_0^{(x,0^+)}}{2}\right)e^{-\frac{x^2}{4it}}\Lambda\Big(\frac{|x|}{2\sqrt{it}}-ik\sqrt{it}\Big) \\
&+\left(\frac{\mu_+^{(x,0^-)}}{\omega_+-ik}+\frac{\mu_-^{(x,0^-)}}{\omega_--ik}+\frac{\mu_0^{(x,0^-)}}{2}\right)e^{-\frac{x^2}{4it}}\Lambda\Big(\frac{|x|}{2\sqrt{it}}+ik\sqrt{it}\Big) \\
&-\sum\limits_{j=\pm}\left(\frac{\mu_j^{(x,0^-)}}{\omega_j-ik}+\frac{\mu_j^{(x,0^+)}}{\omega_j+ik}\right)e^{-\frac{x^2}{4it}}\Lambda\Big(\frac{|x|}{2\sqrt{it}}+\omega_j\sqrt{it}\Big) \\
&+e^{ikx-ik^2t},
\end{split}
\end{equation}
using the coefficients $\mu_0$, $\mu_\pm$, and $\omega_\pm$ from Theorem \ref{satz_Green_function}.
In the special case $k=0$ this formula is understood in the sense that $\mu_j^{(x,0^\pm)}/\omega_j=0$, whenever $\omega_j=0$.
\end{prop}

\begin{proof}
We start by calculating the functions $\Psi_j(t,x,e^{ik\,\cdot\,})$ for $j\in\{0,1,\text{free}\}$ from \eqref{Eq_Psi}. Since the holomorphic continuation $F(z)=e^{ikz}$ of the initial condition satisfies the assumption (\ref{Eq_Exponential_boundedness}) for the special choice $\alpha=\frac{\pi}{4}$, we can use the absolute convergent integral representation (\ref{Eq_Psi3}). For the functions $\Psi_0$ and $\Psi_\text{free}$ we now use the integral identity (\ref{Eq_Lambda_Integral}) to get
\begin{equation*}
\Psi_0\big(t,x;e^{\pm ik\,\cdot}\big)=\frac{1}{2\sqrt{\pi t}}\int_0^\infty e^{-\frac{(|x|+y\sqrt{i})^2}{4it}}e^{\pm iky\sqrt{i}}dy=\frac{1}{2}e^{-\frac{x^2}{4it}}\Lambda\Big(\frac{|x|}{2\sqrt{it}}\mp ik\sqrt{it}\Big),
\end{equation*}
as well as
\begin{equation*}
\Psi_\text{free}\big(t,x;e^{\pm ik\,\cdot}\big)=\frac{1}{2\sqrt{\pi t}}\int_0^\infty e^{-\frac{(x-y\sqrt{i})^2}{4it}}e^{\pm iky\sqrt{i}}dy=\frac{1}{2}e^{-\frac{x^2}{4it}}\Lambda\Big(\frac{-x}{2\sqrt{it}}\mp ik\sqrt{it}\Big).
\end{equation*}
For the function $\Psi_1$ we use the integral (\ref{Eq_Lambda_Integral_2}) to get, at least for $\omega$ and $k$ not both vanishing, the explicit solution
\begin{align*}
\Psi_1\big(t,x;\omega,e^{\pm ik\,\cdot}\big)&=\sqrt{i}\int_0^\infty\Lambda\Big(\frac{|x|+y\sqrt{i}}{2\sqrt{it}}+\omega\sqrt{it}\Big)e^{-\frac{(|x|+y\sqrt{i})^2}{4it}}e^{\pm iky\sqrt{i}}dy \\
&=\frac{e^{-\frac{x^2}{4it}}}{\omega\pm ik}\left(\Lambda\Big(\frac{|x|}{2\sqrt{it}}\mp ik\sqrt{it}\Big)-\Lambda\Big(\frac{|x|}{2\sqrt{it}}+\omega\sqrt{it}\Big)\right).
\end{align*}
For the very special case $\omega=k=0$ the second integral in \eqref{Eq_Lambda_Integral_2} gives
\begin{equation*}
\Psi_1(t,x;0,1)=\sqrt{i}\int_0^\infty\Lambda\Big(\frac{|x|+y\sqrt{i}}{2\sqrt{it}}\Big)e^{-\frac{(|x|+y\sqrt{i})^2}{4it}}dy
=-\sqrt{it}\,\Lambda'\left(\frac{|x|}{2\sqrt{it}}\right)e^{-\frac{x^2}{4it}};
\end{equation*}
here, however, the precise value of the integral is not needed since $\mu_j^{(x,y)}=0$, $j=\pm$, whenever $\omega_j=0$ in Cases I-III
in Theorem \ref{satz_Green_function}, and thus the corresponding term in the decomposition \eqref{Eq_Psi_decomposition} is absent.

Assembling now all these terms as in the decomposition \eqref{Eq_Psi_decomposition} and using the identity \eqref{Eq_Lambda_Negative}
for the terms involving $\Psi_\text{free}$ gives \eqref{Eq_Plane_wave_solution}.
\end{proof}

In the next theorem the long time asymptotics of the plane wave solution in Proposition~\ref{prop_Plane_wave_solution} is found.
While the exponentially decaying $e^{\omega_j|x|}$-terms in \eqref{Eq_varphi_k_asymptotic} and \eqref{Eq_varphi_k0_asymptotic} are due
to negative bound states (see Remark~\ref{lastrem}),
the oscillating terms $e^{ikx}$ and $e^{i|kx|}$ in the first line of \eqref{Eq_varphi_k_asymptotic}, or their absence in \eqref{Eq_varphi_k0_asymptotic},
show that the solution $\Psi(t,x;e^{ik\,\cdot\,})$ oscillates with frequency $k$.
Therefore, roughly speaking, the sequence $(\Psi(t,x;F_n))_n$ shows the characteristic superoscillatory property
since the functions $\Psi(t,x;F_n)$ oscillate with
the frequencies $k_j(n)$ and the limit function $\Psi(t,x;e^{ik\,\cdot\,})$ oscillates with the larger frequency $k$.

\begin{satz}\label{satz_Asymptotics}
For every $k\in\mathbb{R}\setminus\{0\}$ the solution of the Schr\"odinger equation \eqref{Eq_Schroedinger} with initial condition $F(x)=e^{ikx}$ admits the long time asymptotics
\begin{equation}\label{Eq_varphi_k_asymptotic}
\begin{split}
\Psi(t,x;e^{ik\,\cdot})=&e^{ikx-ik^2t}+2\left(\frac{\mu_+^{(x,-k)}}{\omega_+-i|k|}+\frac{\mu_-^{(x,-k)}}{\omega_--i|k|}+\frac{\mu_0^{(x,-k)}}{2}\right)e^{i|kx|-ik^2t} \\
&-\sum\limits_{j=\pm}\left(\frac{\mu_j^{(x,0^-)}}{\omega_j-ik}+\frac{\mu_j^{(x,0^+)}}{\omega_j+ik}\right)2\Theta(-\omega_j)e^{\omega_j|x|+i\omega_j^2t}+\mathcal{O}\Big(\frac{1}{\sqrt{t}}\Big),
\end{split}
\end{equation}
as $t\to\infty$, using the coefficients $\mu_0$, $\mu_\pm$, and $\omega_\pm$ from Theorem \ref{satz_Green_function}. Moreover, for $k=0$ we get the similar expansion
\begin{equation}\label{Eq_varphi_k0_asymptotic}
\begin{split}
\Psi(t,x;1)=&\frac{\mu_+^{(x,0^+)}+\mu_+^{(x,0^-)}}{\omega_+}+\frac{\mu_-^{(x,0^+)}+\mu_-^{(x,0^-)}}{\omega_-}+\frac{\mu_0^{(x,0^+)}+\mu_0^{(x,0^-)}}{2}+1 \\
&-\sum\limits_{j=\pm}\frac{\mu_j^{(x,0^-)}+\mu_j^{(x,0^+)}}{\omega_j}2\Theta(-\omega_j)e^{\omega_j|x|+i\omega_j^2t}+\mathcal{O}\Big(\frac{1}{\sqrt{t}}\Big),
\end{split}
\end{equation}
as $t\rightarrow\infty$. The formula \eqref{Eq_varphi_k0_asymptotic} is understood in the sense that $\mu_j^{(x,0^\pm)}/\omega_j= 0$, whenever $\omega_j=0$.
\end{satz}

\begin{proof}[Proof of Theorem \ref{satz_Asymptotics}]
In the explicit solution \eqref{Eq_Plane_wave_solution} we can use the asymptotic expansion \eqref{Eq_Lambda_Asymptotic} of the function $\Lambda$,
to get for every $k,\omega_j\in\mathbb{R}\setminus\{0\}$
\begin{align*}
\Lambda\Big(\frac{|x|}{2\sqrt{it}}\pm ik\sqrt{it}\Big)&=2\Theta(\pm k)e^{\big(\frac{|x|}{2\sqrt{it}}\pm ik\sqrt{it}\big)^2}+\mathcal{O}\Big(\frac{1}{\sqrt{t}}\Big),\quad\text{as }t\rightarrow\infty, \\
\Lambda\Big(\frac{|x|}{2\sqrt{it}}+\omega_j\sqrt{it}\Big)&=2\Theta(-\omega_j)e^{\big(\frac{|x|}{2\sqrt{it}}+\omega_j\sqrt{it}\big)^2}+\mathcal{O}\Big(\frac{1}{\sqrt{t}}\Big),\quad\text{as }t\rightarrow\infty.
\end{align*}
Note, that all the terms in \eqref{Eq_Plane_wave_solution} with $\omega_j=0$ vanish since in this case also $\mu_j^{(x,0^{\pm})}=0$ by its definition in Theorem \ref{satz_Green_function}. Hence we can use the above asymptotics to get the long time behaviour
\begin{align*}
\Psi(t,x;e^{ik\,\cdot})=&\left(\frac{\mu_+^{(x,0^+)}}{\omega_++ik}+\frac{\mu_-^{(x,0^+)}}{\omega_-+ik}+\frac{\mu_0^{(x,0^+)}}{2}\right)2\Theta(-k)e^{-ik|x|-ik^2t} \\
&+\left(\frac{\mu_+^{(x,0^-)}}{\omega_+-ik}+\frac{\mu_-^{(x,0^-)}}{\omega_--ik}+\frac{\mu_0^{(x,0^-)}}{2}\right)2\Theta(k)e^{ik|x|-ik^2t} \\
&-\sum\limits_{j=\pm}\left(\frac{\mu_j^{(x,0^-)}}{\omega_j-ik}+\frac{\mu_j^{(x,0^+)}}{\omega_j+ik}\right)2\Theta(-\omega_j)e^{\omega_j|x|+i\omega_j^2t} \\
&+e^{ikx-ik^2t}+\mathcal{O}\Big(\frac{1}{\sqrt{t}}\Big),\quad\text{as }t\rightarrow\infty,
\end{align*}
which easily simplifies to (\ref{Eq_varphi_k_asymptotic}). For $k=0$ we get from \eqref{Eq_Plane_wave_solution} the representation
\begin{align*}
\Psi(t,x;1)=&\left(\frac{\mu_+^{(x,0^+)}+\mu_+^{(x,0^-)}}{\omega_+}+\frac{\mu_-^{(x,0^+)}+\mu_-^{(x,0^-)}}{\omega_-}+\frac{\mu_0^{(x,0^+)}+\mu_0^{(x,0^-)}}{2}\right)e^{-\frac{x^2}{4it}}\Lambda\Big(\frac{|x|}{2\sqrt{it}}\Big) \\
&-\sum\limits_{j=\pm}\left(\frac{\mu_j^{(x,0^-)}}{\omega_j}+\frac{\mu_j^{(x,0^+)}}{\omega_j}\right)e^{-\frac{x^2}{4it}}\Lambda\Big(\frac{|x|}{2\sqrt{it}}+\omega_j\sqrt{it}\Big)+1.
\end{align*}
Using the Taylor series $\erf(z)=\frac{2}{\sqrt{\pi}}\sum\limits_{n=0}^\infty\frac{(-1)^n}{n!(2n+1)}z^{2n+1}$ we get the asymptotics
\begin{equation*}
e^{-\frac{x^2}{4it}}\Lambda\Big(\frac{|x|}{2\sqrt{it}}\Big)=1-\erf\Big(\frac{|x|}{2\sqrt{it}}\Big)=1+\mathcal{O}\Big(\frac{1}{\sqrt{t}}\Big).
\end{equation*}
Hence the wave function $\Psi(t,x;1)$ reduces to \eqref{Eq_varphi_k0_asymptotic} in the limit $t\rightarrow\infty$.
\end{proof}

\begin{bem}\label{lastrem}
We note that the $e^{\omega_j|x|}$-terms, $j=\pm$, in the asymptotics in
\eqref{Eq_varphi_k_asymptotic} and \eqref{Eq_varphi_k0_asymptotic} correspond to negative bound states of the underlying self-adjoint
Schr\"odinger operator.
In fact, a bound state corresponding to the eigenvalue (energy) $E\in\mathbb{R}$ is a function $\psi\in L^2(\mathbb{R})$ which satisfies
\begin{subequations}\label{Eq_Bound_state}
\begin{align}
-\frac{\partial^2}{\partial x^2}\psi(x)&=E\,\psi(x),\qquad x\in\mathbb{R}\setminus\{0\}, \label{Eq_Bound_state_1} \\
(I-J)\vvect{\psi(0^+)}{\psi(0^-)}&=i(I+J)\vvect{\frac{\partial}{\partial x}\psi(0^+)}{-\frac{\partial}{\partial x}\psi(0^-)}. \label{Eq_Bound_state_2}
\end{align}
\end{subequations}
In order to get a non-trivial $L^2$-solution of the differential equation \eqref{Eq_Bound_state_1} we need $E<0$; in this case the general solution is given by
\begin{equation*}
\psi(x)=\left\{\begin{array}{ll} A\,e^{-x\sqrt{-E}}, & x>0, \\ B\,e^{x\sqrt{-E}}, & x<0, \end{array}\right.
\end{equation*}
for some constants $A,B\in\mathbb{C}$.
Plugging the limits $\psi(0^\pm)$ and $\frac{\partial}{\partial x}\psi(0^\pm)$ into the jump condition leads to the linear system of equations
\begin{equation}\label{Eq_Bound_state_linear_system}
(I-J)\vvect{A}{B}=-i\sqrt{-E}\,(I+J)\vvect{A}{B}.
\end{equation}
A direct calculation using the matrix \eqref{Eq_General_unitary_matrix} and the property $|\alpha|^2+|\beta|^2=1$ of the matrix entries shows
\begin{equation*}
 \begin{split}
  &\det\big((I-J)+i\sqrt{-E}(I+J)\big)\\
   &\qquad=\det\left(\mmatrix{1-\alpha e^{i\phi}}{\bar\beta e^{i\phi}}{-\beta e^{i\phi}}{1-\bar\alpha e^{i\phi}}+
    i\sqrt{-E}\mmatrix{1+\alpha e^{i\phi}}{-\bar\beta e^{i\phi}}{\beta e^{i\phi}}{1+\bar\alpha e^{i\phi}}\right) \\
   &\qquad=\big(1-2\Re(\alpha)e^{i\phi}+e^{2i\phi}\big)+2i\big(1-e^{2i\phi}\big)\sqrt{-E}-\big(1+2\Re(\alpha)e^{i\phi}+e^{2i\phi}\big)\big(\sqrt{-E}\big)^2 \\
   &\qquad =2e^{i\phi}\Big(\big(\cos(\phi)-\Re(\alpha)\big)+2\sin(\phi)\sqrt{-E}-\big(\cos(\phi)+\Re(\alpha)\big)\big(\sqrt{-E}\big)^2\Big)
 \end{split}
\end{equation*}
and for $E<0$ this determinant vanishes if and only if
\begin{equation*}
 \sqrt{-E}=\begin{cases} \frac{\sin(\phi)\mp\sqrt{1-\Re(\alpha)^2}}{\cos(\phi)+\Re(\alpha)},& \Re(\alpha)\not=-\cos(\phi),\\
            -\cot(\phi),& \Re(\alpha)=-\cos(\phi)\not=-1.
           \end{cases}
\end{equation*}
When comparing with the three different cases
in Theorem \ref{satz_Green_function} we see that $\sqrt{-E}=-\omega_\pm$ with $\omega_\pm<0$  in Case I and $\sqrt{-E}=-\omega_+$ with $\omega_+<0$ in Case II
lead to negative eigenvalues.
More precisely, if $\omega:=\omega_+=\omega_-<0$, then $E=-\omega^2$ is an eigenvalue of multiplicity two with linear independent eigenfunctions
\begin{equation}\label{ws}
\psi_1(x)=e^{\omega|x|}\quad\text{and}\quad\psi_2(x)=\sgn(x)e^{\omega|x|}.
\end{equation}
If $\omega_+\neq\omega_-$, then each $\omega_\pm<0$ leads to an eigenvalue $E_\pm=-\omega_\pm^2$ of multiplicity one with corresponding eigenfunction
\begin{equation}\label{wa}
\psi(x)=\left(1\mp\frac{\Im(\beta)+\sgn(x)\big(\Im(\alpha)+i\Re(\beta)\big)}{\sqrt{1-\Re(\alpha)^2}}\right)e^{\omega_\pm|x|};
\end{equation}
we leave it to the reader to check that the function in \eqref{ws} and \eqref{wa} satisfy the interface condition \eqref{Eq_Bound_state_linear_system}.

\begin{figure}
\begin{tikzpicture}[domain=0:pi/2]
\draw[thick,->] (4.2,0)--(5,0);
\draw[thick,->] (0,4.2)--(0,4.7);
\filldraw[black] (4.8,0) node[anchor=north] {$\Re(\alpha)$};
\filldraw[black] (0,4.5) node[anchor=west] {$\phi$};
\filldraw[black] (-4.1,4) node[anchor=east] {$\pi$};
\filldraw[black] (4,-0.1) node[anchor=north] {$1$};
\filldraw[black] (-4.1,0) node[anchor=east] {$0$};
\filldraw[black] (-4,-0.1) node[anchor=north] {$-1$};

\draw[ultra thick,-] (-4,4)--(-4,0)--(4,0)--(4,4);
\draw[ultra thick,dashed] (-4,4)--(4,4);

\filldraw[black] (0,0.8) node[anchor=center] {$\omega_+<0,\,\omega_-\geq 0$};
\filldraw[black] (0,3.2) node[anchor=center] {$\omega_+\geq 0,\,\omega_-<0$};
\filldraw[black] (2.5,2) node[anchor=center] {$\omega_\pm<0$};
\filldraw[black] (-2.5,2) node[anchor=center] {$\omega_\pm\geq 0$};

\filldraw[black] (-4.3,2.7) node[anchor=east] {$\Re(\alpha)=\cos(\phi)$};
\filldraw[black] (4.3,2.7) node[anchor=west] {$\Re(\alpha)=-\cos(\phi)$};
\draw (-4.3,2.7)--(-3.1,3.1);
\draw (4.3,2.7)--(3.1,3.1);

\draw[very thick,samples=150] plot ({-3.98*cos(\x r)},4/pi*\x+0.02);
\draw[very thick,dashed,samples=150] plot ({-3.98*cos(\x r)},4/pi*\x-0.02);
\draw[very thick,dashed,samples=150] plot ({3.98*cos(\x r)},4/pi*\x+0.02);
\draw[very thick,samples=150] plot ({3.98*cos(\x r)},4/pi*\x-0.02);
\draw[very thick,samples=150] plot ({3.98*cos(\x r)},4-4/pi*\x+0.02);
\draw[very thick,dashed,samples=150] plot ({3.98*cos(\x r)},4-4/pi*\x-0.02);
\draw[very thick,dashed,samples=150] plot ({-3.98*cos(\x r)},4-4/pi*\x+0.02);
\draw[very thick,samples=150] plot ({-3.98*cos(\x r)},4-4/pi*\x-0.02);
\end{tikzpicture}
\caption{Possible negative eigenvalues $\omega_\pm<0$ of the Schr\"{o}dinger operator
depending on the choice of $\phi\in [0,\pi)$ and $\Re(\alpha)\in [-1,1]$ in the matrix $J$.
The continuous/dashed lines illustrate boundaries that do/don't belong to the parameter regions.
Case III in Theorem~\ref{satz_Green_function} corresponds to the left lower corner $(-1,0)$, Case II is depicted by the curve
$\Re(\alpha)=-\cos(\phi)$, and the remaining points constitute Case I.}
\end{figure}
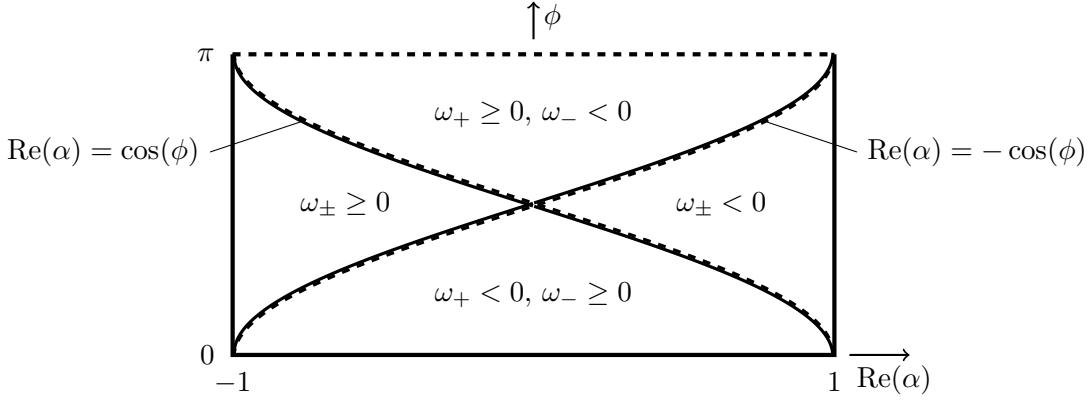
\end{bem}

We conclude this section with an example illustrating Proposition~\ref{prop_Plane_wave_solution} and Theorem~\ref{satz_Asymptotics}
for the important special case of $\delta$- and $\delta'$-potentials; cf. \cite[Theorem 3.2]{ABCS19}.

\begin{bsp}\label{deltadelatp}
Using the coefficients $\mu_0$, $\mu_\pm$ and $\omega_\pm$ from Example \ref{bsp_Delta}
and Example \ref{bsp_Deltaprime} it follows that for $k\in\mathbb R$ the plane wave solutions for the $\delta$- and $\delta'$-interaction are given by
\begin{align*}
\Psi_\delta(t,x;e^{ik\,\cdot})=&\left(-\frac{c}{2(c+ik)}\Lambda\Big(\frac{|x|}{2\sqrt{it}}-ik\sqrt{it}\Big)-\frac{c}{2(c-ik)}\Lambda\Big(\frac{|x|}{2\sqrt{it}}+ik\sqrt{it}\Big)\right. \\
&\quad\left.+\frac{c^2}{c^2+k^2}\Lambda\Big(\frac{|x|}{2\sqrt{it}}+c\sqrt{it}\Big)\right)e^{-\frac{x^2}{4it}}+e^{ikx-ik^2t}, \\
\Psi_{\delta'}(t,x;e^{ik\,\cdot})=&\left(\frac{ik\sgn(x)}{2(c+ik)}\Lambda\Big(\frac{|x|}{2\sqrt{it}}-ik\sqrt{it}\Big)+\frac{ik\sgn(x)}{2(c-ik)}\Lambda\Big(\frac{|x|}{2\sqrt{it}}+ik\sqrt{it}\Big)\right. \\
&\quad\left.-\frac{ikc\sgn(x)}{c^2+k^2}\Lambda\Big(\frac{|x|}{2\sqrt{it}}+c\sqrt{it}\Big)\right)+e^{ikx-ik^2t},
\end{align*}
and for $k\in\mathbb{R}\setminus\{0\}$ their asymptotics as $t\rightarrow\infty$ are
\begin{align*}
\Psi_\delta(t,x;e^{ik\,\cdot})&=e^{ikx-ik^2t}-\frac{c}{c-i|k|}e^{i|kx|-ik^2t}+\Theta(-c)\frac{2c^2}{c^2+k^2}e^{c|x|+ic^2t}+\mathcal{O}\Big(\frac{1}{\sqrt{t}}\Big), \\
\Psi_{\delta'}(t,x;e^{ik\,\cdot})&=e^{ikx-ik^2t}+\frac{ik\sgn(x)}{c-i|k|}e^{i|kx|-ik^2t}-\Theta(-c)\frac{2ick\sgn(x)}{c^2+k^2}e^{c|x|+ic^2t}+\mathcal{O}\Big(\frac{1}{\sqrt{t}}\Big).
\end{align*}
\end{bsp}

\end{document}